\documentclass[11pt]{article}%
\usepackage{amssymb}
\usepackage[all]{xy}
\usepackage[utf8]{inputenc}
\usepackage{amsfonts,dsfont}
\usepackage{amsmath,amsthm}
\usepackage{amssymb}
\usepackage{url}
\usepackage{eurosym}
\usepackage{hyperref}%
\usepackage{graphics}
\usepackage{graphicx}
\usepackage{thmtools}
\usepackage{thm-restate}
\usepackage[ruled,vlined]{algorithm2e}
\usepackage{epsfig}
\usepackage{subfigure}
\providecommand{\U}[1]{\protect\rule{.1in}{.1in}}
\usepackage{color}
\usepackage{enumerate}

\usepackage[square,sort,comma,numbers]{natbib}

\def\X{\mathcal{X}}

\def\Y{\mathcal{Y}}

\def\S{\mathcal{S}}

\def\R{\mathbb{R}}
\def\RR{\mathbb{R}}

\def\NN{\mathbb{N}}

\def\conv{\textrm{conv}}
\def\convhull{\overline{\conv}}
\def\argmin{\textrm{arg}\min}

\def\ri{\textrm{ri}}

\def\epi{\textrm{epi}}

\def\tf{\tilde{f}}
\def\tfi{\tilde{f}_{i}}

\def\bfi{\bar{f}_{i}}

\def\tXi{\tilde{\X}_{i}}
\def\tX{\tilde \X}
\def\tx{\tilde x}
\def\txi{\tilde{x}_{i}}

\def\tri{\tilde{r}_{i}}

\def\tGamma{\tilde{\Gamma}}
\def\bx{\bar x}
\def\bxi{\bar{x}_{i}}
\def\bGamma{\bar{\Gamma}}
\def\tr{\tilde{r}}
\def\tri{\tilde{r}_{i}}
\newtheorem{thm}{Theorem}[section]
\newtheorem{ass}{Assumption}

\newtheorem{lem}[thm]{Lemma}

\newtheorem{prop}[thm]{Proposition}
\theoremstyle{definition}
\newtheorem{defn}[thm]{Definition}
\newtheorem{rem}{Remark}[section]
\newtheorem*{nota}{Notation}
\numberwithin{equation}{section}
\definecolor{dgreen}{rgb}{0.00,0.49,0.00}
\definecolor{Brown}{rgb}{0.45,0.0,0.05}
\newcommand{\tcbl}[1]{\textcolor{black}{#1}}

\begin{document}
\title{Approximate Nash equilibria in large non-convex\\ aggregative games}
\author{Kang Liu\footnote{CMAP, Ecole Polytechnique; Inria Saclay; kang.liu@polytechnique.edu }, ~~Nadia Oudjane\footnote{EDF R\&D and  FIME (Finance for Energy Market Research Centre); nadia.oudjane@edf.fr}, ~~Cheng Wan\footnote{EDF R\&D and  FIME (Finance for Energy Market Research Centre); cheng.wan.05@polytechnique.org (corresponding author)}}

\maketitle
\begin{abstract}
This paper shows the existence of $\mathcal{O}(\frac{1}{n^\gamma})$-Nash equilibria in $n$-player noncooperative sum-aggregative games in which the players' cost functions, depending only on their own action and the average of all players' actions, are lower semicontinuous in the former while $\gamma$-H\"{o}lder continuous in the latter. Neither the action sets nor the cost functions need to be convex. For an important class of sum-aggregative games, which includes congestion games with $\gamma$ equal to 1, a gradient-proximal algorithm is used to construct $\mathcal{O}(\frac{1}{n})$-Nash equilibria with at most $\mathcal{O}(n^3)$ iterations. These results are applied to a numerical example concerning the demand-side management of an electricity system. The asymptotic performance of the algorithm when $n$ tends to infinity is illustrated.
\end{abstract}

\medskip

\textbf{MSC Class} {Primary: 91A06; secondary: 90C26}
%

\section{Introduction}\label{sec:introdution}

This paper studies approximate pure-Nash equilibria (PNE) for $n$-player noncooperative games involving non-convexities in players’ costs or action sets. The goal is to show the existence of such approximate equilibria under certain conditions and to propose an algorithm that allows them to be calculated effectively in some specific cases. In particular, this paper focuses on a specific class of noncooperative games (which includes congestion games) referred to as sum-aggregative games (Selten \cite{Selten1970}, Corch\'on \cite{Corchon1994}, Jensen \cite{Jensen2010}). The cost of each player depends on the weighted sum of the other players' decisions. These games have practical applications in various aspects of political science, economics, social biology, and engineering, such as voting \cite{PalfreyRosenthal1983,MyersonWeber1993}, market competition \cite{MurphyAl1982}, public goods provision \cite{BasileAl2016,FoucartWan2018}, rent seeking \cite{PerezDavidVerdier1992}, population dynamics \cite{HofbauerSigmund1998}, traffic analysis \cite{Dafermos1980,MarcottePatricksson2007}, communications network control \cite{OrdaRomShimkin1993,LibmanOrda2001} and electrical system management \cite{Horta2018,Jaquotetal2020}. 
However, in these real-life situations, the players' action sets and their cost functions are often \emph{non-convex}. This paper is actually motivated by concrete applications for which it is unreasonable to neglect the \emph{non-convexities} inherent to the problem. In particular, we are interested in demand-side management in electrical systems \cite{Paulin}, where each flexible consumer is considered as a player trying to minimize her electricity bill by modulating her consumption (e.g., electric vehicle charging) which is subject to non-convex constraints. 
\smallskip


In the convex framework,  a PNE is known to exist under mild regularity conditions (see, for example, Rosen \cite{Rosen1965}). Outside the convex framework, it is  generally difficult to provide \emph{existence} results for PNEs and \emph{approximation algorithms} with performance guarantees. 
%
Our work addresses these two issues and makes the following contributions.
\smallskip

\noindent (i) Theoretically, Proposition \ref{thm:existepsilon} and Theorem \ref{thm:main} show the  existence of $\mathcal{O}(\frac{1}{n^\gamma})$-PNEs for $n$-player non-convex sum-aggregative games in which the players' cost functions 
are $\gamma$-H\"{o}lder continuous with respect to the aggregate. Neither the action sets nor the cost functions need to be convex.
\smallskip

\noindent (ii) Algorithmically,  
in the specific case of congestion games in which the cost functions are Lipschitz continuous with respect to the aggregate (i.e., $\gamma=1$),
we present an iterative gradient-proximal algorithm to compute an $\mathcal{O}(\frac{1}{n})$-PNE of the original non-convex game within at most $\mathcal{O}(n^3)$ iterations, according to Theorem \ref{thm:congestionPNE}. 
%

\smallskip

\noindent (iii) Practically, the usefulness of this approach is demonstrated in  Section \ref{sec:numerical}, where a numerical simulation with the gradient-proximal algorithm is performed for a demand-side management problem involving flexible electric vehicle charging. 
\smallskip


The originality of this paper lies in the circumvention of the non-convexity through the exploitation 
the fact that large sum-aggregative games approach a convex framework when the number of players is large. 
The counterpart of this approach is  to search for an $\epsilon$-PNE (cf. Definition \ref{def:epsilonNash}) instead of an exact PNE. The main inspirations for the present work are \cite{Starr1969} (in economics) and \cite{Wang2017} (in optimization). Starr \cite{Starr1969} was interested in computing general equilibria for a non-convex competitive economy in terms of price and quantity, while Wang \cite{Wang2017} considered large-scale non-convex separable optimization problems coupled by sum-aggregative terms. In both cases, the authors proposed to convexify the problem, taking advantage of the large number of agents or subproblems
to bound the error induced by the convexification thanks to the Shapley-Folkman Lemma (cf. Lemma \ref{lm:SF}). 
Roughly speaking, the Shapley-Folkman Lemma states that the Minkowski sum of a finite number of sets in a Euclidean space is close to convex when the number of sets is very large compared with their common  dimension. 
%
Our first contribution consists of two novelties. First, in Proposition \ref{thm:existepsilon}, we extend their approach to non-convex sum-aggregative games to show the existence of an $\epsilon$-PNE.
Second, in Theorem \ref{thm:main} we show that one can also construct an $\epsilon$-PNE of the non-convex game 
from an $\epsilon$-PNE 
of the auxiliary convexified game provided that some \textit{stability condition} is satisfied. 
This second novelty is more significant, and it is crucial for our algorithmic contribution. 
Our second contribution consists in proposing an algorithm returning an $\epsilon$-PNE of the convexified game. 
This $\epsilon$-PNE verifies the stability condition that allows us to recover an $\epsilon$-PNE of the original non-convex game.
\smallskip

\paragraph{Related works.} 
%
The existence of PNEs is not guaranteed in non-convex games, except in some particular cases. 
For example, for games in which players have a finite number of actions, the existence of PNE is known for Rosenthal's congestion games \cite{Rosenthal1973} and some other specific class of congestion games \cite{Tran2011,Meyers2006}.  
For games with discrete (but not necessarily finitely many) strategies, Sagratella \cite{Sagratella2016} 
proved the existence of PNEs for a particular class of such games and proposed an algorithm leading to one of the equilibria. 
However, when the players' cost functions are non-convex and/or their action sets are non-convex but not necessarily finite, there is no general result for the existence of PNEs. 

\smallskip


Concerning algorithms for the computation of ($\epsilon$-)PNE, there are few results. The existing results are almost restricted to some special cases in the convex setting.  
A common approach is to solve the variational inequality characterizing the PNEs (cf. Facchinei and Pang \cite{FacchineiPang2003} and the references therein). 
Scutari et al. \cite{ScutariAl2014} considered generic $n$-player games that need not be large nor aggregative but must have a strongly monotone inequality characterizing the PNE. They used proximal best-reply algorithms to solve this variational inequality. 
Paccagnan, Kamgarpour and Lygeros \cite{PaccagnanKamgarpourLygeros2016}  considered a specific convex aggregative game and used a decentralized gradient projection algorithm to solve the strongly monotone variational inequality characterizing the PNE. 
Paccagnan et al. \cite{PaccagnanAl2019} studied $\epsilon$-PNEs in large convex aggregative games with coupling constraints. Their methodology is close to ours in the sense that they only look for an $\epsilon$-PNE (which they call the Wardrop equilibrium) instead of an exact PNE. 
They used, respectively, a decentralized gradient projection algorithm and a decentralized best-reply (to the aggregate term) algorithm to solve the variational inequality characterizing this Wardrop equilibrium.
\smallskip

Finally, the Shapley-Folkman lemma has been extensively applied in non-convex optimization for its convexification effect. Aubin and Ekeland \cite{AubinEkeland1976} used the lemma to derive an upper bound for the duality gap in an additive, separable non-convex optimization problem. Since then, quite a few papers have extended or sharpened this result (cf. Ekeland and Temam \cite{EkelandTemam1999}, Bertsekas and coauthors \cite{Bertsekas1979,BertsekasSandell1982}, Pappalardo \cite{Pappalardo1986}, Kerdreux et al. \cite{KerdreuxAl2019}, Bi and Tang \cite{BiTang2020}). These theoretical results have applications in engineering problems, such as the large-scale unit commitment problem \cite{LauerAl1982,BertsekasAl1983}, the  optimization of plug-in electric vehicle charging \cite{VujanicAl2016}, the optimization of multicarrier communication systems \cite{YuLui2006}, supply-chain management \cite{VujanicAl2014}, and spatial graphical model estimation \cite{FangLiuWang2019}.

\paragraph{Organization.} Section \ref{sec:exist} focuses on the existence of $\epsilon$-PNE in large non-convex aggregative games. Section \ref{sec:congestion} presents an algorithm to compute such an $\epsilon$-equilibrium for congestion games. A numerical application in Section \ref{sec:numerical} illustrates the usefulness of our results. Section \ref{sec:perspectives} concludes.
\color{black}

\paragraph{Notation.} In a Euclidean space, $\| \cdot \|$ denotes the $l^2$-norm. For a point $x\in \RR^d$ and a subset $\X$ of $\RR^d$, $d(x, \X):=\inf_{y\in \X} \{\|x-y\|\}$ is the distance from the point to the subset. For two subsets $\X$ and $\Y$ of $\RR^d$, their Minkowski sum is the set $\{x+y \, | \, x\in \X, y\in \Y\}$. For $x\in \RR^d$ and $r\in \RR^{+}$, $B(x,r):=\{y\in \RR^{d} \,|\, \|y-x\|\leq r\}$, with the $r$-radial ball centered on $x$. 

For a matrix $A\in \RR^{d} \times \RR^{q}$, $\| \cdot \|_2$ is the 2--norm of the matrix: $\|A\|_2:=\sqrt{\lambda_{\max}(A^\tau A)}$ where $A^\tau$ is the transpose of $A$ and $\lambda_{max}(A^\tau A)$ stands for the largest eigenvalue of the matrix $A^\tau A$.
\medskip

The proof of Proposition \ref{thm:existNEconv} and the lemmata used for the proof are given in Appendix A.  All the other proofs and intermediate results are contained in Appendix B.

\section{Existence of $\epsilon$-PNE in large non-convex sum-aggregative games}\label{sec:exist}
\subsection{A non-convex sum-aggregative game and its convexification}
Consider an $n$-player noncooperative game $\Gamma$. The players are indexed over $N=\{1,2,\cdots, n\}$. Each player $i\in N$ has an action set $\X_i\subset \R^d$, which is closed and bounded but not necessarily convex. Let $\tXi := \conv(\X_i)$ be the convex hull of $\X_i$ (which is also closed and bounded) and denote $\X:=\prod_{i\in N} \X_i$, $\tX:= \prod_{i\in N} \tXi $, $\tX_{-i} := \prod_{j\in N_{-i}}\tX_j$, where $N_{-i}:=N\setminus\{i\}$. Let  the constant  $\Delta>0$  be such that, for all $i\in N$, the compact set $\X_i$ has a diameter $\vert \X_i\vert := \max_{x_i, y_i \in \X_i}\|x_i - y_i\|$ that is not greater than $\Delta$. 

As usual, let $x_{-i}$ denote the profile of actions of all the players except player $i$. Each player $i$ has a real-valued cost function $f_i$ defined on $\X_i \times \tX_{-i}$, which has the following specific form: 
\begin{equation}
\label{eq:cost}
f_i(x_i, x_{-i}):=\theta_i \bigg(x_i,\, \frac{1}{n}\sum_{j\in N} A_j x_j \bigg)\ ,\quad\textrm{for any}\ x_i \in \X_i\, , \ x_{-i} \in \tX_{-i}\ ,
\end{equation}
where each $A_j$ is a $q \times d$ matrix for all $j\in N$, and $\theta_i$ is a real-valued function defined on $\X_i \times \Omega$, with $\Omega\subset \RR^q$ being a neighborhood of $\{ \frac{1}{n} \sum_{j\in N} A_j y_j\, \vert\, y_j \in \tX_j ,\, \forall j\in N\}$. 

Let the constant $M>0$ be such that $\| A_i \|_{2} \leq M$ for each $i\in N$.
\medskip

\tcbl{The game $\Gamma$ is a sum-aggregative game because each player's cost function depends on her own action and an aggregate of all the players' actions.}

\begin{defn}[$\epsilon$-pure Nash equilibrium]\label{def:epsilonNash}
For a constant $\epsilon \geq 0$, an \emph{$\epsilon$-pure Nash equilibrium ($\epsilon$-PNE)}  $x^\epsilon\in \X$ in game $\Gamma$ is a profile of actions of the $n$ players such that, for each player $i \in N$,
\begin{equation*}
f_i(x^\epsilon_i, x^\epsilon_{-i}) \leq f_i(x_i, x^\epsilon_{-i}) + \epsilon \, , \textrm{ for any } x_i \in \X_i\, .
\end{equation*}
If $\epsilon=0$, then $x^\epsilon$ is a \emph{pure Nash equilibrium (PNE)}.
\end{defn}

\tcbl{This definition of $\epsilon$-PNE corresponds to the notion of \emph{additively} 
$\epsilon$-PNE in the literature.}
\medskip

For non-convex games (in which either action sets or cost functions are not convex), the existence of a PNE is not \tcbl{established for the general cases}. 
This paper uses an auxiliary convexified version of the non-convex game, which is helpful both in the proof of the existence of an $\epsilon$-PNE of the non-convex game and in the construction of such an approximate PNE.

\begin{defn}[Convexified game and generators]\label{def:w_i}
The \emph{convexified game} $\tGamma$ associated with $\Gamma$ is a noncooperative game played by  $n$ players. Each player $i\in N$ has an action set $\tXi$ and a real-valued cost function $\tfi$ defined on $\tX$ as follows: for all $x\in \tX$,
\begin{equation}
\label{eq:costConv}
\tfi(x_i, x_{-i})
=\inf_{(\alpha^k)_{k=1}^{d+1}\in \S_{d}; \, (z^k)_{k=1}^{d+1} \in \X_i^{d+1} } \bigg\{\sum_{k=1}^{d+1} \alpha^k f_i(z^k,x_{-i})\,\Big\vert \, x_i=\sum_{k=1}^{d+1} \alpha^k z^k \, \bigg\} \, ,
\end{equation}
where $\mathcal{S}_{d}:=\{\alpha=(\alpha^k)_{k=1}^{d+1} \in\RR^{d+1} \, \vert\, \forall k\,, \alpha^k\geq 0\,, \, \sum_{k=1}^{d+1} \alpha^k=1\}$ denotes the probability simplex of dimension $d$.

\tcbl{For all $i\in N$ and $x\in \tX$, let a minimizer in \eqref{eq:costConv}
be generically denoted by $(\alpha(i,x), z(i, x)) \in \mathcal{S}_d \times \X_i^{d+1}$. For such a vector $z(i,x)$,  we denote the set of its $d+1$ components by $Z(i,x)$ and call it a \emph{generator for $(i,x)$.}} 

\end{defn}

%
%

For a lower semicontinuous (l.s.c.) function, equation \eqref{eq:costConv} just defines its convex hull (cf. Lemma \ref{lm1}). This particular form of definition is proposed in \cite{Bertsekas1996}.
\smallskip

PNEs and $\epsilon$-PNEs for the convexified game $\tGamma$ are defined in the same way as for the game $\Gamma$ in Definition \ref{def:epsilonNash}. 
\bigskip

The remainder of this subsection is dedicated to a preliminary analysis of the convexified game.

First let us introduce an assumption \tcbl{on the functions $\theta_i$'s characterizing the cost functions according to \eqref{eq:cost}}. It ensures the existence of generators for all $(i,x)\in N\times \X$ (cf. Lemma \ref{lm1} for a proof).

\begin{ass}\label{ass:lsc} ~~ \\
\noindent (1) 
For any player $i\in N$, for any $y\in \Omega \tcbl{\subset \R^q}$, the function $x_i \mapsto \theta_i(x_i, y)$ is l.s.c. on $\X_i$.

\noindent (2) There exist constants $H>0, \gamma>0$ such that,
for all $i\in N$,  for all $x_i\in \X_i$, the function $y\mapsto \theta_i (x_i,y)$ is $(H,\gamma)$-H\"older continuous on $y\in \Omega$, i.e., 
\begin{equation}
\label{eq:ass:theta}
\vert \theta_i(x_i,y')-\theta_i(x_i,y)\vert \leq H \Vert y'-y\Vert^\gamma\ .
\end{equation} 
\end{ass}

\begin{rem}
 It is straightforward from Assumption \ref{ass:lsc} that $f_i(\cdot, x_{-i})$ is l.s.c. in $x_i\in \X_i$ for any fixed $x_{-i}\in \tX_{-i}$.
\end{rem}

According to Lemma \ref{lm1}, $x_i \mapsto \tfi(x_i, x_{-i})$ is convex and l.s.c on $\tXi$. Its subdifferential exists; let it be denoted  by $\partial_i \tfi(\cdot, x_{-i})$. Then, for each $x_i \in \tXi$, $\partial_i \tfi(x_i, x_{-i})$ is a nonempty convex subset of $\R^d$.

\begin{prop}[Existence of PNE in $\tGamma$]\label{thm:existNEconv}
Under Assumption \ref{ass:lsc}, the convexified game $\tGamma$ admits a PNE.
\end{prop}
\begin{proof}
This results from Theorem \ref{thm:lsc} in Appendix A.
\end{proof}

\begin{rem}
\tcbl{Theorem \ref{thm:lsc} is a natural extension of Rosen's theorem on the existence of PNEs in games with convex continuous cost functions \cite{Rosen1965} to the case where the cost functions are only l.s.c. instead of being continuous with respect to the players' own actions. } 
\smallskip

The following example shows that even the continuity of $f_i$ on $\X_i$ cannot guarantee the continuity of $\tfi$ on $\tXi$, meaning that Rosen's theorem is not sufficient here.

Consider $d=3$, $\X_i=T\cup B \cup S$ where $T=\{(x^1,x^2,x^3) \in \RR^3 \vert (x^1)^2+(x^2)^2=1, x^3=1\}$, $B=\{(x^1,x^2,x^3) \in \RR^3 \vert (x^1)^2+(x^2)^2=1, x^3=-1\}$, and $S=\{(x^1,x^2,x^3) \in \RR^3 \vert x^1=1, x^2=0, -1\leq x^3 \leq 1\}$; $f_i$ is independent of $x_{-i}$, and $f_i(x)=0$ for $x\in T\cup B$, $f_i(x)=|x^3|-1$ for $x\in S$. Then, for all $x\in \{(x^1,x^2,x^3) \in \RR^3 \vert (x^1)^2+(x^2)^2=1, x^3=0\} \subset \partial \tXi$, $\tfi(x)=0$ except for $x^*=(1,0,0)$, but $\tfi(x^*)=f_i(x^*)=-1$.
\end{rem}

\subsection{Existence and construction of an $\epsilon$-PNE of the non-convex game}

The following proposition shows the existence of an $\epsilon$-PNE in the non-convex game $\Gamma$ and its construction from an exact PNE of the convexified game $\tGamma$. \tcbl{In particular, we observe that $\epsilon$ is small when the number of players $n$ is large with respect to $q$, the dimension of the space in which the aggregate $\frac{1}{n} \sum_{i\in N} A_i x_i$ lies.}

\begin{prop}[Existence of $\epsilon$-PNE]	\label{thm:existepsilon}
Under Assumption \ref{ass:lsc}, the non-convex game $\Gamma$ admits an $\epsilon$-PNE, where $\epsilon = 2 H (\frac{(\sqrt{q}+1)M\Delta}{n})^\gamma$.  

In particular, 
suppose that $\tx\in \tX$ is a PNE in $\tGamma$ (which exists according to Proposition \ref{thm:existNEconv}), and $Z(i,\tx)$ is an arbitrary generator for each player $i$; then, $x^{*}\in \X$ such that
\begin{equation}\label{eq:argmin}
x^{*}\in \underset{x_i\in Z(i,\tx),\, i\in N }{\operatorname{argmin}} \Big\|\sum_{i\in N} A_i \tx_i-\sum_{i\in N} A_i x_i\Big \|^2\, ,
\end{equation}
is an $\epsilon$-PNE of the non-convex game $\Gamma$.
	%
\end{prop}

\emph{Sketch of the proof:}   By the definition of the PNE in $\tGamma$, $\tx_i$ is a best response to $\tx_{-i}$ in terms of $\tf_i$. By Lemma \ref{lm1}, all the points in $Z(i,\tx)$ are also best replies to $\tx_{-i}$ in terms of $f_i$.

 We then use the Shapley-Folkman Lemma (Lemma \ref{lm:SF}) to disaggregate $\frac{1}{n}\sum_i A_i \tx_i$ over the sets $Z(i,\tx)$ to obtain a feasible profile $x^*$.
Finally, we can show that $\frac{1}{n}\sum_i A_i x^*_i \approx \frac{1}{n}\sum_i A_i \tx_i$ and that $x^*_i$ is (almost) a best response to $\frac{1}{n}\sum_i A_i x^*_i$.
\medskip

From an algorithmic point of view, a PNE is not always easy or fast to compute for the convexified game $\tGamma$, even though its existence is guaranteed. 
Even when we have a convergent algorithm, the outputs of the algorithm at each iteration provide only approximations of the exact PNE which may constitute $\epsilon$-PNEs but rarely exact PNEs. 
Then, the question that naturally arises is whether the idea above is still valid if $\tx$ is only an $\epsilon$-PNE of $\tGamma$, i.e., $\tx_i$ is an $\epsilon$-best response to $\tx_{-i}$ in terms of $\tf_i$. 
The answer is yes if the $\epsilon$-PNE $\tx$ of the convexified game $\tGamma$ satisfies a more demanding condition, introduced by the following definition.

\begin{defn}[Stability condition]\label{def:approxeq}
In game $\tGamma$, 
for a given $\eta\geq 0$, 
a point $\tx\in \tX$ is said to satisfy the \emph{$\eta$-stability condition with respect to} $(Z(i,\tx))_i$ 
if, for each player $i$, 
 $\tf_i(x_i,\tx_{-i})  \leq \tfi(\tx_i,\tx_{-i}) + \eta$ for all $x_i \in Z(i,\tx)$.

A point $\tx$ is said to satisfy the \emph{$\eta$-stability condition} if it satisfies the $\eta$-stability condition with respect to a certain generator profile $(Z(i,\tx))_i$.

A point $\tx$ is said to satisfy the \emph{full $\eta$-stability condition} if it satisfies the $\eta$-stability condition with respect to any generator profile $(Z(i,\tx))_i$.
\end{defn}

The stability condition of $\tx$ with respect to $(Z(i,\tx))_i$ means that the cost for each player $i$ is only slightly increased if the player's choice is unilaterally perturbed within the convex hull of the generator $Z(i,\tx)$.

According to Lemma \ref{lm1}, a PNE of $\tGamma$ satisfies the  full $0$-stability condition.
\medskip

A sufficient condition for the $\eta$-stability of $\tx$ is given by Lemma \ref{lem:approxeq}.
\begin{lem}\label{lem:approxeq}
Under Assumption \ref{ass:lsc}, for any action profile $\tx\in \tX$, for any player $i$,  if there is a generator $Z(i,\tx)$ and $h\in \partial_i \tfi(\txi, \tx_{-i})$ such that 
\begin{equation}\label{eq:firstorder}
 \bigl\langle   h, x_i - \tx_i \bigr\rangle  \geq  -\eta \|x_i -\txi \|\, ,\quad \forall x_i \in \conv \, Z(i,\tx)\, ,
\end{equation}
then,
\begin{equation*}
  |\tf_i(x_i,\tx_{-i}) - \tfi(\tx_i,\tx_{-i}) | \leq \eta\|x_i -\txi \| \, , \quad \text{for all }\,  x_i \in Z(i,\tx)\ .
\end{equation*}
In particular, $\tx$ satisfies the $\eta \Delta$-stability condition  with respect to $(Z(i,\tx))_i$.
\end{lem}


The following theorem describes how to construct an $\mathcal{O}(\frac{1}{n^\gamma})$-PNE of the original non-convex game $\Gamma$ when we know both an $\epsilon$-PNE of the convexified game $\tGamma$ satisfying the $\eta$-stability condition and the associated generator profile.

\begin{thm}[Construction of $\epsilon$-PNE of $\Gamma$]	\label{thm:main}
Under Assumption \ref{ass:lsc}, suppose that $\tx\in \tX$ is an $\epsilon$-PNE in $\tGamma$ that satisfies the $\eta$-stability condition with respect to a specific generator profile $(Z(i,\tx))_i$. Let $x^{*}\in \X$ be such that
\begin{equation}\label{eq:argmin}
x^{*}\in \underset{x_i\in Z(i,\tx),\, i\in N }{\operatorname{argmin}} \Big\|\sum_{i\in N} A_i \tx_i-\sum_{i\in N} A_i x_i\Big \|^2\, .
\end{equation}
Then, $x^*$ is an $\tilde \epsilon$-PNE of the non-convex game $\Gamma$, where $\tilde \epsilon = \epsilon+\eta + 2 H (\frac{(\sqrt{q}+1)M\Delta}{n})^\gamma$. 
\end{thm}

\subsection{A distributed randomized ``Shapley-Folkman disaggregation''}\label{subsec:randomSF}
Once an exact PNE or an $\epsilon$-PNE $\tx$ satisfying the $\eta$-stability condition of the convexified game $\tGamma$ is obtained, as well as the associated generator profile $(Z(i,\tx))_i$, we would like to find an $\tilde{\epsilon}$-PNE of the non-convex game $\Gamma$ whose existence is shown by Proposition \ref{thm:existepsilon} and Theorem \ref{thm:main}. However, solving \eqref{eq:argmin} is generally hard (cf. Udell and Boyd \cite{UdellBoyd2016} for such a ``Shapley-Folkman disaggregation'' in a particular  optimization setting). 
In this section, we present a method for computing an $\check{\epsilon}$-mixed-strategy Nash equilibrium (MNE) in a distributed way, based on the known $\epsilon$-PNE of $\tGamma$, its associated generator profile, and its coefficients. The algorithm is called ``distributed'' because it computes the mixed-strategy $\mu_i$ of each player $i$ from the information of $\tx_i$, the generator $Z(i, \tx)$ and the corresponding coefficients only. 

%
%

\begin{prop}\label{prop:mixed}
Under Assumption \ref{ass:lsc}, suppose that $\tx\in \tX$ is an $\epsilon$-PNE of $\tGamma$ satisfying the $\eta$-stability condition with respect to $(Z(i,\tx))_i$, and each player $i$ plays a mixed strategy $\tilde{\mu}_i$ independently, i.e., a random action $X_i$ following the distribution $\tilde{\mu}_i$ over $\X_i$, defined by $\mathbb{P}(X_i = x_i^{l}) = \alpha_i^l$, where  $x_i^l \in Z(i,x)$ for $l=1,\ldots, d+1$, and $(\alpha^l_i)_{l=1}^{d+1}=\alpha(i,x)$ is the corresponding vector of coefficients. Then, for $\gamma\leq 1$, $\tilde\mu=(\tilde\mu_i)_i$ is an $\check \epsilon$-MNE of the non-convex game $\Gamma$, where $\check \epsilon = \epsilon+\eta + 2 H (\frac{(\sqrt{n}+1)M\Delta}{n})^\gamma$, in the sense that
\begin{equation*}
\mathbb{E}\big[ f_i(X_i, X_{-i})  \big] \leq \mathbb{E}\big[  f_i(x_i,X_{-i}) \big]  + \check \epsilon\, , \quad \forall \, x_i \in \X_i \ .
\end{equation*}
\end{prop}

\begin{rem}
Note that this is not an ``adaptive'' algorithm that allows the players to attain an $\epsilon$-PNE/MNE in the non-convex game through a decentralized adaptation/learning process; instead, it is a distributed, randomized disaggregation algorithm to recover an $\check{\epsilon}$-MNE of $\Gamma$ from a known $\epsilon$-PNE of $\tGamma$ satisfying the $\eta$-stability condition and its generators. 

Besides, when $\tx$ is an exact PNE of $\tGamma$, all the generators of $\tx$ can be used in this algorithm. On the contrary, when $\tx$ is only an $\epsilon$-PNE of $\tGamma$, we need a specific profile of generators $(Z(i,x))_i$, with respect to which $\tx$ satisfies the $\eta$-stability condition. How to find such a generator profile is not evident. However, in the next section, we provide an algorithm to compute, for a specific class of games, an $\epsilon$-PNE of $\tGamma$ satisfying the full $\eta$-stability condition (i.e., with respect to any profile of generators of $\tx$). In that case, any profile of generators can be used in the algorithm of Proposition \ref{prop:mixed}.

Finally, note that the estimated error $\check \epsilon$ for the distributed randomized approximate MNE in Proposition \ref{prop:mixed} is larger than the estimated error $\tilde \epsilon$ for the approximate PNE 
$x^*$ in 
Theorem \ref{thm:main}.  
\end{rem}

\section{Computing $\epsilon$-equilibria for large non-convex congestion games}\label{sec:congestion}

\subsection{Non-convex  congestion games}
Congestion games are an extensively studied class of sum-aggregative games. In this section, we present an iterative algorithm to compute an $\omega(K,n)\Delta$-PNE of the convexification of a specific congestion game, in which $\omega(K,n)$ tends to zero when both the number of players $n$ and the number of iterations $K$ tend to $+\infty$, while $\frac{n}{K}$ tends to zero. 
Note that any algorithm returning an approximate PNE of the convexified game will not necessarily ensure that it verifies the stability condition. 
The proposed algorithm is of particular interest because we can show that the iterates provide an $\omega(K,n)\Delta$-PNE of the convexified game that satisfies the full $\omega(K,n)\Delta$-stability condition (cf. Proposition \ref{lm:Nash-disagrregative}). 
Then, taking $K\sim \mathcal{O}(n^3)$, Theorem \ref{thm:congestionPNE} shows that one can recover an $\mathcal{O}(\frac{1}{n})$-PNE of the original non-convex congestion game from this $\omega(K,n)\Delta$-PNE of the convexified game. 
\medskip

Consider a congestion game in which each player $i\in N$ has an action set $\X_i\subset \R^d$ and a cost function of the following form:
\begin{equation}
\label{eq:game}
\begin{split}
f_i(x_i,x_{-i}) =& \bigg\langle g \Bigl( \frac{1}{n}  \sum_{j\in N} a_j  x_j\Bigr), x_i \bigg\rangle + h_i \Big( \frac{1}{n} \sum_{j\in N} a_j  x_j \Big) + r_i(x_i)\\
=&\sum_{t=1}^{d} g_{t} \Bigl( \frac{1}{n} \sum_{j\in N} a_j x_{j,t} \Bigr) x_{i,t}  + h_i \Big( \frac{1}{n} \sum_{j\in N} a_j  x_j \Big) +r_i(x_i)\, . 
\end{split}
\end{equation}

Suppose that the following assumptions hold on $\X_i$, $(a_j)_{j\in N}\in\R^n$, $g_t$, $h_i$ and $r_i$. 

\begin{ass} \label{ass:example}~~
\begin{itemize}
\item There exist constants $m>0$ and $M>0$ such that $m\leq a_i\leq M $ for all $i\in N$.
\item For $t=1,\ldots, d$, the function $g_t: \R\rightarrow \R$ is $L_{g_t}$-Lipschitz continuous and nondecreasing on a neighborhood of $[D_1, D_2]$, where the constants $D_1$ and $D_2$ are such that $D_1\leq \min_{t=1, \ldots, d; x\in \tX} \frac{1}{n}\sum_{j\in N} a_j x_{j,t}\leq \max_{t=1, \ldots, d ; x\in \tX} \frac{1}{n}\sum_{j\in N} a_j x_{j,t}\leq D_2$. 
\item For each $i\in N$, the function $h_i: \R^d\rightarrow \R$ is $L_{h_i}$-Lipschitz continuous on $[D_1, D_2]^d$. 
\item Players' local cost functions $r_i: \R^d \rightarrow \R$ are uniformly bounded, i.e., there exists a constant $B_r>0$ such that, for all $i\in N$ and all $x_i\in \X_i$, $|r_i(x_i)|\leq B_r$.
\end{itemize}
\end{ass}

\begin{nota}
Let the constant $\Delta := \max\{\max_{i\in N}\max_{x_i\in \tXi}\|x_i\|, \max_{i\in N} |\tX_i|\}$.	 Let  $L_g:=\linebreak[4]\max_{1\leq t \leq d}L_{g_t}$, $L_{h}:= \max_{i\in N} L_{h_i}$, $B_g := \max_{1\leq t \leq d, D_1\leq s \leq D_2}|g_t (s)|$. 
\end{nota}

\bigskip

The convexification of $\Gamma$ is rather complicated to compute in the general case. Let us first introduce an auxiliary game that is very close to $\Gamma$  whose convexification is easier to obtain.

Fix arbitrarily $x^+_i\in \X_i$ for each player $i\in N$. The auxiliary game $\bGamma$ is defined as follows: the player set and each player's action set are the same as in $\Gamma$, but player $i$'s  cost function is, for all $ x_i \in \X_i$ and all $x_{-i}\in \tX_{-i}$,
\begin{equation*}
\bfi(x_i, x_{-i}):= \Big\langle g \Bigl( \frac{1}{n} \sum_{j\neq i}a_j  x_j + \frac{1}{n} a_i x^+_i\Bigr), x_i \Big\rangle + r_i(x_i)\, .
\end{equation*}

The original game $\Gamma$ can be approximated by the auxiliary game $\bGamma$ because their approximate equilibria are very close to each other, \tcbl{as the following lemma shows (see its proof  in Appendix B).
\begin{lem}\label{lm: approx game}
Under Assumption \ref{ass:example},  for the auxiliary game $\bGamma$,
\begin{enumerate}[(1)]
\item Assumption \ref{ass:lsc} is verified with $H=L_g \Delta$ and $\gamma=1$;
\item an $\epsilon$-PNE of $\bGamma$ is an $(\epsilon +  \frac{L_h M \Delta}{n}+
 \frac{2L_g M\Delta^2}{n})$-PNE of $\Gamma$.
\end{enumerate}
\end{lem}
}

For any fixed $x_{-i}\in \tX_{-i}$, $\bfi(\cdot, x_{-i})$ is composed of a linear function of $x_i$ and a local function of $x_i$. By abuse of notation, let us still use $\bfi$ to denote its convexification on $\tXi$. More explicitly,
\begin{equation}\label{eq:gameconv}
\bfi(x_i, x_{-i}):= \Big\langle g \Bigl(  \frac{1}{n}\sum_{j\neq i} a_j  x_j + \frac{1}{n} a_i x^+_i\Bigr), x_i \Big\rangle + \tri(x_i) \, , 
\end{equation}
where $\tri$ is the convexification of $r_i$ defined on $\tXi$ in the same way as \eqref{eq:costConv}.

By abuse of notation, let $\tGamma$ denote the convexification of $\bGamma$ on $\tX$.

\subsection{A gradient-proximal algorithm}
This subsection presents a gradient-proximal algorithm based on the block coordination proximal algorithm introduced by Xu and Yin \cite{XuandYin2013} that can be used to construct an $\mathcal{O}(\frac{1}{n})$-PNE of $\tGamma$ that satisfies the full $\mathcal{O}(\frac{1}{n})$-stability condition.
\bigskip

\begin{algorithm}[H]\label{algo:BCP}
\SetAlgoLined
\textbf{Initialization:} choose initial point $x^0=(x^0_1, x^0_2,\ldots,x^0_{n})\in \tX$\\
\For{$k=1,2,\cdots$}{
\For{$i=1,2,\ldots,n$}{
\begin{equation}\label{alg:proximal1}
{x}_{i}^{k}=\underset{x_i\in \tX_{i} }{\operatorname{argmin}}\Big\langle  g \Bigl(  \frac{1}{n} \sum_{j < i} a_j  x^{k}_j + \frac{1}{n}\sum_{j \geq i} a_j  x^{k-1}_j \Bigr), x_i-x_{i}^{k-1}\Big\rangle+\frac{a_iL_g}{2n}\big\|x_i-x_{i}^{k-1}\big\|^{2}+\tri(x_i)
\end{equation}
}
\If{stopping criterion is satisfied}
{return $(x^{k}_1, x^{k}_2,\ldots,x^{k}_{n})$. Break.}
}
\caption{Gradient-proximal algorithm for $\tGamma$}
\end{algorithm}

\begin{rem} 
This is a decentralized-coordinated type of algorithm. The coordinator needs to know the current choices of the players and $(a_i)_i$ to compute $g \bigl(  \frac{1}{n} \sum_{j < i} a_j  x^{k}_j + \frac{1}{n}\sum_{j \geq i} a_j  x^{k-1}_j \bigr)$.
 The value of the vector $g \bigl(  \frac{1}{n} \sum_{j < i} a_j  x^{k}_j + \frac{1}{n}\sum_{j \geq i} a_j  x^{k-1}_j \bigr)$ is sent to player $i$ by the coordinator in iteration $k$, when it is that player's turn to compute. No detailed information concerning the other players' choices is revealed. Receiving this value, player $i$ uses her local information, i.e., $a_i$, $\tri$ and $\tX_{i}$, to update her choice according to \eqref{alg:proximal1} and then sends it to the coordinator. \end{rem}

\begin{prop}\label{lm:Nash-disagrregative}
Under Assumption \ref{ass:example}, 
for $K\in \NN^*$, there is $k^*\leq K$ such that $x^{k^*}$ is an $\omega(K,n)\Delta$-PNE of game $\tGamma$ that satisfies the full $\omega(K,n)\Delta$-stability condition, where 
\begin{equation}\label{eq:estimation}
\omega(K, n)= \frac{\sqrt{2CL_g}M}{m}\sqrt{\frac{n}{K}} + \frac{ 2  L_g M \Delta}{n}\, ,
\end{equation}
where $C=(d\Delta L_g + 2 B_r) M$.

In particular, if the constant $K\geq \frac{2C}{m^2 L_g} n^{1+2\delta}+1$ for some constant $\delta>0$, then there exists some $k^{*}\leq K$ such that $x^{k^{*}}$ is an $L_g M \Delta(n^{-\delta}+2 \Delta n^{-1})$-PNE of game $\tGamma$ satisfying the full $L_g M \Delta(n^{-\delta}+2 \Delta n^{-1})$-stability condition.
\end{prop}

\begin{thm}\label{thm:congestionPNE}
Under Assumption \ref{ass:example}, for a constant $\delta>0$ and integer $K\geq \frac{2C}{m^2 L_g} n^{1+2\delta}+1$, let $x^{*}\in \X$ be the pure-strategy profile generated by \eqref{eq:argmin}, where $\tx$ is replaced by $x^{k^{*}}$ in Proposition \ref{lm:Nash-disagrregative}. Then, $x^{*}$ is a $\big(2L_g M \Delta \big( n^{-\delta}+ \frac{ (\sqrt{q}+4)\Delta}{n} \big)+\frac{L_h M\Delta}{n} \big)$-PNE of the non-convex game $\Gamma$. 
\end{thm}

For the case in which a ``Shapley-Folkman'' disaggregation of $x^{k^*}$ is not easy to obtain, one can use the distributed randomized disaggregation method introduced in Section \ref{subsec:randomSF} to immediately obtain an $\check{\epsilon}$-MNE, where $\check{\epsilon}$ is given by the following corollary. However, the quality of approximation is worse than that of a ``Shapley-Folkman'' disaggregation.

\begin{prop}\label{cor:congestionMNE}
Under Assumption \ref{ass:example}, for a constant $\delta>0$ and integer $K\geq \frac{2C}{m^2 L_g} n^{1+2\delta}+1$, let $\tilde\mu=(\tilde\mu_i)_i$ be a profile of independent mixed strategies defined as in Lemma \ref{lm:decomposeSF}, where $\tx$ is replaced by $x^{k^{*}}$ in Proposition \ref{lm:Nash-disagrregative}. Then, $\tilde{\mu}$ is a $\big(2L_g M \Delta \big( n^{-\delta}+ \frac{ (\sqrt{n}+4)\Delta}{n} \big)+\frac{L_h M\Delta}{n} \big)$-MNE of the non-convex game $\Gamma$. 
\end{prop}

\section{Numerical example}\label{sec:numerical}
In this section, we consider an example of flexible electric vehicle charging control; the convex version of this problem was studied by Jacquot et al. \cite{Paulin}.
%
%
%
\tcbl{Each player must charge the battery of her electric vehicle when she arrives home after work. A player’s cost is defined in the form of \eqref{eq:cost} so that it depends on both her own consumption and the aggregate consumption of all the players. Such a design is intended to ensure that 
the Nash equilibria or $\epsilon$-Nash equilibria attain the goal of decreasing the peak demand and smoothing the load curve of the power grid.
In this context, the technical constraints of the battery of an electric vehicle 
limit the number of feasible consumption profiles. They also generally imply non-convex action sets; for example, they only allow for discrete power consumption profiles. 
}

\tcbl{More specifically,} one day is divided into peak hours (e.g., 6 am--10 pm) and off-peak hours. The electricity production cost function for total flexible loads of $\ell^{P}$ and $\ell^{O\!P}$ at peak and off-peak hours are, respectively, $C^P(\ell^P)=\alpha^P_0 \ell^P+\beta_0 (\ell^P)^2$ and $C^{O\!P}(\ell^{O\!P})=\alpha^{O\!P}_0 \ell^{O\!P}+\beta_0 (\ell^{O\!P})^2$, 
where $ \alpha^P_0>\alpha^{O\!P}_0>0$ and $\beta_0>0$.
Player $i$'s action is denoted by $\ell_i=(\ell_i^{P},\ell_i^{O\!P})$,  where $\ell_i^{P}$ (resp. $\ell_i^{O\!P}$) is the peak (resp. off-peak) consumption of player $i$. Player $i$'s electricity bill is then defined by
\begin{equation*}
b_i(\ell_i,\ell_{-i}):=\frac{C^P(\ell^P)}{\ell^P}\ell_i^P+\frac{C^{O\!P}(\ell^{O\!P})}{\ell^{O\!P}}\ell_i^{O\!P}\, ,
\end{equation*}
where $\ell^P = \sum_i \ell_i^{P}$, and $\ell^{O\!P} = \sum_i \ell_i^{O\!P}$. 
Player $i$'s cost is then defined by 
\begin{equation}
\label{eq:PaulinGame0}
\phi_i(\ell_i,\ell_{-i})=b_i(\ell_i,\ell_{-i})+\gamma_i \Vert \ell_i-\ell^{ref}_i\Vert^2
\end{equation}
where $\gamma_i$ indicates the player's sensitivity to the deviation from her preference $\ell^{ref}$. In \cite{Paulin}, the action set of player $i$ is the convex compact set $S_i = \{ \ell_i=(\ell_i^{P},\ell_i^{O\!P}) \, \vert \, \ell_i^{P}+\ell_i^{O\!P}=e_i, \underline{\ell_i^{P}} \leq \ell_i^{P} \leq \overline{\ell_i^{P}},\underline{\ell_i^{O\!P}} \leq \ell_i^{O\!P} \leq \overline{\ell_i^{O\!P}}\}$, where $e_i$ stands for the energy required by player $i$ to charge an electric vehicle battery and $\underline{\ell_i^{P}}$ and $\overline{\ell_i^{P}}$ (resp. $\underline{\ell_i^{O\!P}}$ and $\overline{\ell_i^{O\!P}}$) are the minimum and maximum power consumption for player $i$ during peak (resp. off-peak) hours. However, for various reasons, such as finite choices for charging power or battery protection guidelines that indicate that the charging must be interrupted as infrequently as possible, the players' action sets can be non-convex. For example, in this paper a particular case with the non-convex action set $S^{N\!C}_i = \{ \ell_i=(\ell_i^{P},\ell_i^{O\!P}) \, \vert\, \ell_i^{P}+\ell_i^{O\!P}=e_i,   \ell_i^{P} \in \{\underline{\ell_i^{P}},\overline{\ell_i^{P}}\}\}$ is adopted for numerical simulation.
\smallskip


Let us apply Algorithm \ref{algo:BCP} to this game. The asymptotic performance of the algorithm for large $n$ is illustrated. 
%
%
%

First, game \eqref{eq:PaulinGame0} is reformulated with uni-dimensional actions. For simplification, suppose that all the players have the same type of electric vehicle (EV), a 2018 Nissan Leaf, with a battery capacity $e$, and two charging rate levels $p_{\min}$ and $p_{\max}$. The total consumption of player $i$ is denoted by $e_i$ and determined by a parameter $\tau_i$ as follows: $e_i=(1-\tau_i)e=\ell_i^{P}+\ell_i^{O\!P}$, where $\tau_i\in [0,1]$ signifies the remaining proportion of energy in the player's battery when  she arrives at home. Let $x_i :=\frac{\ell_i^{P}}{e}$ denote player $i$'s strategy in the following reformulation of game \eqref{eq:PaulinGame0}:
\begin{equation}
\label{eq:PaulinGame}
\tilde f^{(n)}_i(x_i,x_{-i})=\tilde b^{(n)}_i(x_i,x_{-i})+\tilde \gamma_i \Vert x_i-x^{ref}_i\Vert^2\ ,
\end{equation}
where $\tilde{\gamma}_i$ indicates how much player $i$ cares about deviating from her preferred consumption profile and is uniformly set to be $ne$ for simplification, and
\begin{align*}
	\tilde b^{(n)}_i(x_i,x_{-i}) & =\, (\alpha_0^P+\beta_0 n e \frac{1}{n}\sum_j (1-\tau_j)x_j)\ell_i^P+(\alpha_0^{O\!P}+\beta_0 ne \frac{1}{n}\sum_j(1-\tau_j) (1-x_j))\ell_i^{O\!P}\\ 
&  = \, e (1-\tau_i)\Big [
	\big(\alpha_0^P-\alpha_0^{O\!P}-\beta_0 n e+2\beta_0 n e \frac{1}{n}\sum_j(1-\tau_j)x_j \big)x_i \\
& ~~~ +\alpha_0^{O\!P}+\beta_0 n e-\beta_0 n e\frac{1}{n}\sum_j(1-\tau_j)x_j\Big ] \ .
\end{align*}

The non-convex action set of player $i$, introduced in Section \ref{sec:introdution} as $S^{N\!C}_i = \{ \ell_i=(\ell_i^{P},\ell_i^{O\!P}) \, \vert\, \ell_i^{P}+\ell_i^{O\!P}=e_i,   \ell_i^{P} \in \{\underline{\ell_i^{P}},\overline{\ell_i^{P}}\}\}$, is now translated into $\X_i=\{\underline{x}_i, \overline{x}_i\} \subset [0,1]$, where $\underline{x}_i$ and $\overline{x}_i$ correspond, respectively, to charging at $p_{min}$ and $p_{max}$. 
%

By extracting the common factor $ne(1-\tau_i)$, player $i$'s cost function becomes
\begin{equation}\label{eq: example}
f^{(n)}_{i}\left(x_{i}, x_{-i}\right):=\bigg\langle g^{(n)}\Big(\frac{1}{n} \sum_{j=1}^{n}(1-\tau_j)  x_{j}\Big), x_{i}\bigg\rangle+h^{(n)}\Big(\frac{1}{n} \sum_{j=1}^{n} (1-\tau_j) x_{j}\Big)+\frac{r^{(n)}_{i}\left(x_{i}\right)}{1-\tau_i}\, ,
\end{equation}
 where $g^{(n)}(y) :=\frac{\alpha_{0}^{P}-\alpha_{0}^{O\!P}}{n} + \beta_0  e(2 y-1)$, $h^{(n)}(y):= \frac{\alpha_{0}^{O\!P}}{n}+\beta_0 e(1- y)$, and $r^{(n)}_{i}(y):= \|y-x^{ref}_i\|^{2}$ for $y\in \RR$, where $\alpha_{0}^{P}=-4.17+0.59\times 12n$ (\euro{}/kWh), $\alpha_{0}^{O\!P}=-4.17+0.59\times 8n$  (\euro{}/kWh), and $\beta_0=0.295$ (\euro{}/kWh$^2$) according to Jacquot et al. \cite{Paulin}.


\paragraph{Simulation parameters}
The peak hours are between 6 am and 10 pm, while the remaining hours of the day are off-peak hours. The battery capacity of a 2018 Nissan Leaf is $e = 40$ kWh. 
The discrete action set of player $i$ is determined as follows. The players' arrival times at home are independently generated according to a \textit{Von Mises} distribution with $\kappa=1$ between 5 pm and 7 pm. Their departure times are independently generated according to a \textit{Von Mises} distribution  with $\kappa=1$ between 7 am and 9 am. The proportion $\tau_i$ of energy in the battery when a player arrives at home  is independently generated according to a \textit{Beta} distribution with the parameter $\beta(2,5)$. Once a player arrives at home, she starts charging at one of the two available levels, $p_{min}=3.7$ kW or $p_{max}=7$ kW. 
This power level is maintained until the energy requirement $e_i$ is reached. The arrival and departure time parameters are defined such that the problem is always feasible, i.e., the energy requirement $e_i$ can always be reached during the charging period by choosing the power level $p_{max}$. 
Players are all assumed to prefer to charge their vehicle as fast as possible, so that $x_i^{ref}=\overline{x}_i$ for all $i$. Fifty instances of the problem are considered for the numerical test. They are obtained by independent simulations of the aforementioned  parameters (players' arrival and departure times and remaining energy when they arrive at home).

Algorithm \ref{algo:BCP} is applied to the EV charging game $\Gamma^{(n)}$ \eqref{eq: example} for $n=2^s$, $s=1,\ldots, 15$. For each game $\Gamma^{(n)}$, for each iteration $k$ of the algorithm, let $x^{(n),k}$ denote the $k^{th}$ iterate of Algorithm~\ref{algo:BCP} applied to game $\Gamma^{(n)}$. Then, the relative error $\epsilon^{(n),k}$ of $x^{(n),k}$ is given by
\begin{align*}
\epsilon^{(n),k} := \min \bigg\{\epsilon \geq 0 \, \Big\vert \, & f^{(n)}_i(x^{(n),k}_i, x^{(n),k}_{-i}) - \inf_{x_i\in \X_i}  f^{(n)}_i(x_i, x^{(n),k}_{-i})  \\
& \leq \epsilon \big(\sup_{x_i\in \X_i} f^{(n)}_i(x_i, x^{(n),k}_{-i}) - \inf_{x_i\in \X_i}  f^{(n)}_i(x_i, x^{(n),k}_{-i})\big) \bigg\}\ .
\end{align*}


\begin{figure}[h!]
	\centering
	{
		\includegraphics[width=1\textwidth]{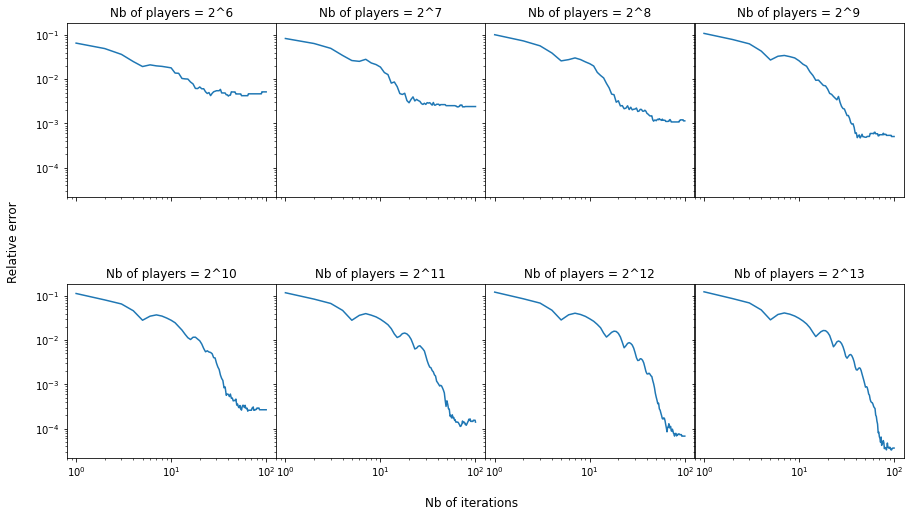}
	}
	\caption{Log-log chart of relative error $\epsilon^{(n),k}$ (averaged over fifty instances of the problem) as a function of the number of iterations $k$ (for a fixed number of players $n=2^6, 2^7, \ldots, 2^{13}$).}
	\label{figs}
\end{figure}
\begin{figure}[h!]
	\centering
	{
		\includegraphics[width=1\textwidth]{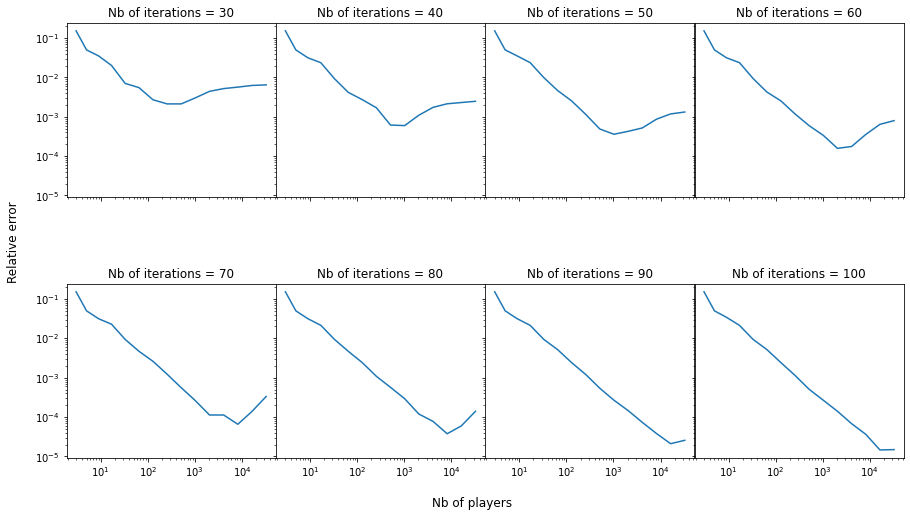}
	}
	\caption{Log-log chart of relative error $\epsilon^{(n),k}$ (averaged over fifty instances of the problem) as a function of the number of players $n$ (for a fixed number of iterations $k=30,40,\ldots,90,100$).}
	\label{figs2}
\end{figure}

As Figure \ref{figs} shows, the relative error decreases with the number of iterations to a certain limit. This limiting relative error decreases with the number of players $n$. This observation is consistent with equation \eqref{eq:estimation} in Proposition \ref{lm:Nash-disagrregative}. For Figure \ref{figs2},  according to Proposition \ref{lm:Nash-disagrregative}, when the iteration number $k$ is fixed, due to the domination of the term $\frac{2 L_g M \Delta}{n}$ in equation \eqref{eq:estimation} when $n$ is small, $\epsilon^{(n),k}$ first decreases linearly with $n$ before reaching a certain threshold. After that,  $\frac{\sqrt{2CL_g}M}{m}\sqrt{\frac{n}{k}} $ dominates the relative error value so that $\epsilon^{(n),k}$ may increase with $n$. The threshold itself increases with the iteration number $k$. This is exactly what Figure \ref{figs2} shows.

\section{Conclusion and perspectives}\label{sec:perspectives}

\tcbl{This paper developed an original approach for the study of non-convex games. Non-convexities are widely present in real applications, and they are known to add nontrivial difficulties in the analysis of existence and computation of equilibria. 
Our approach is restricted to large aggregative games because it is based on the Shapley-Folkman Lemma which essentially exploits the aggregative form with a large number of players. This category covers nevertheless a broad class of games with practical interest, including congestion games.
In particular, we illustrated the relevance of this approach with an industrial application to the coordination of electric vehicle charging. }

\paragraph{Distributed and randomized ``Shapley-Folkman disaggregation''.} In Section \ref{subsec:randomSF}, a distributed disaggregating method is introduced to obtain a randomized ``Shapley-Folkman disaggregation'' for the case $\gamma \leq 1$. It is extremely fast and easy to carry out: once an \tcbl{$\epsilon$-PNE $\tx$ is obtained for the convexified game, as well as the profile of generators $(Z(i,\tx))_i$, each player $i$ randomly chooses one feasible action that is in $Z(i,\tx)$}, according to the distribution law $\alpha(i,x)$. This procedure returns an $\mathcal{O}(\frac{1}{\sqrt{n}^\gamma})$-MNE, with the error vanishing when the number of players goes to infinity. However, even if an $\mathcal{O}(\frac{1}{n^\gamma})$-PNE can be difficult to obtain using an exact ``Shapley-Folkman disaggregation'' especially if a large centralized program is involved, for example, to solve \eqref{eq:argmin}, it would be desirable to find other algorithms that can find better approximations of the Nash equilibria of a non-convex game. Distributed and randomized algorithms are appealing because they can be faster to carry out, they require less coordination and hence are more tractable, and they take advantage of the law of large numbers when $n$ is large. 

\paragraph{Aggregation and disaggregation of clusters.} In a power grid management setting, flexible agents can be regrouped into clusters, and each cluster is commanded by a so-called aggregator. The EV charging game considered in this paper then takes place between the relatively few aggregators instead of the individuals. This “aggregate game” is different from the EV charging game described in this paper, as the individuals are no longer autonomous but are commanded by their respective aggregators rather than choosing their own charging behaviors. One can build an aggregate model for each aggregator by defining his action set as the set of the aggregate actions of the individuals in his cluster and his cost function as an aggregate of the individuals’ costs. When the clusters are large, it is possible to show, with the help of the Shapley-Folkman Lemma, that the aggregators’ action sets and cost functions are almost convex. Then, the game admits an $\epsilon$-PNE (via Rosen’s existence theorem), and its computation could be relatively easier owing to the small number of players. However, each aggregator then has to reconstruct for each individual under his control a feasible action consistent with their aggregate action at the equilibrium of this “aggregate game.” When the constraints of each flexible individual are non-convex, this aggregation/disaggregation approach can be rather difficult to implement.  An original technique based on the Shapley-Folkman Lemma is proposed in Hreinsson et al. \cite{HreinssonAl2021} within the optimization framework, with applications to the management of consumption flexibilities in power systems.  

\paragraph{Acknowledgments.} We are grateful to J. Fr\'ed\'eric Bonnans and Rebecca Jeffers for stimulating discussions and comments. We particularly thank the reviewers and the associated editor for their very relevant remarks, which have helped greatly to improve the paper.

\section*{Appendix A: PNE in l.s.c.$\!$ convex games }
\tcbl{Since we have not found a specific reference of the extension of Rosen's theorem to the l.s.c. case, we prefer to provide our own proof for the sake of completeness.}

\begin{lem}\label{lm:rosen}
Let $R$ be a nonempty convex compact set in $\RR^{n}$. If the real-valued function $\rho(x,y)$ defined on $R\times R$ is continuous in $x$ on $R$ for any fixed $y$ in $R$, l.s.c. in $(x,y)$ on $R\times R$, and convex in $y$ on $R$ for any fixed $x$ in $R$, then the set-valued map $\zeta: R\rightarrow R$, $ x\mapsto \zeta(x) = \argmin_{z\in R}\rho(x,z)$ has a fixed point.	
\end{lem}
\begin{proof}
Kakutani's fixed-point theorem \cite{Kakutani1941} will  be applied for the proof. First, let us show that $\zeta$ is a Kakutani map, i.e., (i) $\Gamma$ is upper semicontinuous (u.s.c.) in the set map sense and (ii) for all $x\in R$, $\zeta(x)$ is non-empty, compact and convex.
\smallskip

\noindent	(i) Fix $x\in R$. On the one hand, since $\rho(x,y)$ is convex w.r.t $y$, $\zeta(x)$ is convex. On the other hand, $\rho(x,y)$ is l.s.c in $y$, while $R$ is compact; hence $\rho(x,y)$ can attain its minimum w.r.t $y$ and $\zeta(x)$ is thus nonempty. Besides, since $\rho$ is l.s.c., $\zeta(x)=\{y|\rho(x,y)\leq \min_{z\in R}\rho(x,z)\}$ is a closed subset of compact set $R$; hence it is compact.
\smallskip
	
\noindent	(ii) Recall that the set-valued map $\zeta$ is u.s.c. if, for any open set $w\subset R$, set $\{x\in R|\, \zeta(x)\subset w\}$ is open. 

Let us first show by contradiction that, for arbitrary $x_0\in R$, for any $ \epsilon>0$, there exists $\delta>0$ such that for all $ z\in B(x_0,\delta),\, \zeta(z)\subset\zeta(x_0)+B(0,\epsilon)$. If this is not true, then there exists $\epsilon_0>0$ and, for all $n\in \NN^*$, the point $z_{n}\in B(x_0, \frac{1}{n})$ such that there exists $y_{n}\in\zeta(z_{n})$ with $d(y_{n},\zeta(x_0))>\epsilon_0$. Since the sequence $\{y_{n}\}$ is in the compact set $R$, it has a subsequence $y_{\phi(n)}$ converging to some $\bar{y}$ in $R$, and $d(\bar{y},\zeta(x_0))\geq \epsilon_0$. Then, for all $y\in R$,
\begin{equation*}
\rho(x_0,\bar{y})\leq\varliminf_{n \rightarrow \infty}\rho(z_{\phi(n)},y_{\phi(n)})\leq\varliminf_{n \rightarrow \infty}\rho(z_{\phi(n)},y)=\rho(x_0,y)\, ,
\end{equation*}
where the first inequality is due to the lower semicontinuity of $\rho$ in $(x,y)$, the second inequality is due to the definition of $\zeta(z_{\phi(n)})$, and the third equality is due to the continuity of $\rho$ in $x$. This shows that $\bar{y}\in \zeta(x_0)$, in contradiction with the fact that $d(\bar{y},\zeta(x_0))\geq \epsilon_0$.
	
Now fix arbitrarily an open set $w\subset R$ and some $x_0\in R$ such that $\zeta({x_0})\subset w$. Since $\zeta({x_0})$ is compact while $w$ is open, there exists $\epsilon>0$ such that $\zeta(x_0)+B(0,\epsilon)\subset w$. According to the result of the previous paragraph, for this particular $\epsilon$, there exists $ \delta>0$ such that $\zeta(z)\subset\zeta(x_0)+B(0,\epsilon)\subset w$ for all $z\in B(x_0,\delta)$. This means that $B(x_0,\delta)\subset \{x\in R\, |\, \zeta(x)\subset w\}$. As a result, the set $\{x\in R\, |\, \zeta(x)\subset w\}$ is open.
	
Finally, according to Kakutani's fixed-point theorem, there exists $\tx\in R$ such that $ \tx\in \zeta(\tx)$.
\end{proof}	

\begin{defn}
A family of real-valued functions $\{f(\cdot, y) : \X \rightarrow \RR \,|\, y\in \Y\}$ indexed by $y$, with $\X \subset \R^{d_1}$ and $\Y\subset \R^{d_2}$, is \emph{uniformly equicontinuous} if, for all $\epsilon>0$, there exists $\delta$ such that, for all $y\in \Y$, $ \|f(x_1,y)-f(x_2,y)\|\leq \epsilon$ whenever $\|x_1-x_2\|\leq \delta$.
\end{defn}

\begin{thm}[Existence of PNE in l.s.c. convex games]\label{thm:lsc}
In an $n$-player game $\Gamma$, if, for each player $i\in \{1, \ldots, n\}$,
\begin{enumerate}[(1)]
\item the action set $\X_i$ is a convex compact subset of $\R^d$,
\item the cost function $f_i(x_i,x_{-i}): \X_i \times \prod_{j\neq i} \X_j  \rightarrow \RR$ is convex and l.s.c.$\!\!$ in $x_i\in \X_i$ for any fixed $x_{-i}\in \prod_{j\neq i} \X_j $, and
\item 
the family of functions $\{f_i(x_i, \cdot): \prod_{j\neq i} \X_j \rightarrow \RR \, |\, x_{i}\in \X_i \}$ are uniformly equicontinuous,
\end{enumerate}
 then $\Gamma$ admits a PNE.
\end{thm}
\begin{proof}
Define the function $\rho(x,y): \X \times \X \rightarrow \R$ by $\rho(x,y) = \sum_{i=1}^{n} f_i(y_i,x_{-i})$, where $\X = \prod_i \X_i$. It is easy to see that a fixed point of the set-valued map $\zeta: \X \rightarrow \X$, $ x\mapsto \zeta(x) = \argmin_{z\in R}\rho(x,z)$ is a Nash equilibrium of game $\Gamma$.

In order to apply Lemma \ref{lm:rosen}, one needs to show the following: (i) $\rho(x,y) $ is continuous in $x$ for each fixed $y$; (ii) $\rho(x,y) $ is l.s.c. in $(x,y)$; (iii) $\rho(x,y) $ is convex in $y$ for each fixed $x$.

Results (i) and (iii) follow straightforwardly from the definition of $\rho$.

For (ii), first note that, by the uniform equicontinuity of  $\{f_i(x_i, \cdot): \prod_{j\neq i} \X_j \rightarrow \RR \, |\, x_{i}\in \X_i \}$  for each $i$ and the fact that $n$ is finite, $\{\rho(\cdot, y), y\in R\}$ is uniformly equicontinuous.  Let $(x^k,y^k)$ be a sequence in $\X \times \X$ indexed by $k$ that converges to $(x,y) \in \X \times \X$. Then,
\begin{equation*}
\begin{split}
\varliminf_{k\rightarrow \infty} (\rho(x^k,y^k)-\rho(x,y)) = &\varliminf_{k\rightarrow \infty}(\rho(x^k,y^k) - \rho(x,y^k) + \rho(x,y^k)- \rho(x,y))\\
= & \varliminf_{k\rightarrow \infty}(\rho(x,y^k)- \rho(x,y))\\
\geq & 0 \, ,
\end{split}
\end{equation*}
where the second equality is due to the uniform equicontinuity of $\{\rho(\cdot,y), y\in \X\}$ and the last inequality is because $\rho(x, y)$ is l.s.c. in $y$ for any fixed $x$.
\end{proof}

\begin{rem}
The property (3) is weaker than the condition that  $f_i$ is continuous on $\X$. Indeed, since $\X$ is compact, $f_i(x_i,x_{-i})$  is uniformly continuous on $\X_i \times \prod_{j\neq i} \X_j$ which implies the equicontinuity of $\{f_i(x_i, \cdot): \prod_{j\neq i} \X_j \rightarrow \RR \, |\, x_{i}\in \X_i \}$. In other words, Rosen's theorem on the existence of convex continuous games with compact convex actions sets is a corollary of Theorem \ref{thm:lsc}.
\end{rem}

\section*{Appendix B: Other proofs and lemmata}

\begin{lem}[Shapley-Folkman Lemma \cite{Starr1969}]\label{lm:SF}
For $n$ compact subsets $S_1, \ldots, S_n$ of $\RR^q$,  let $x \in \conv \sum^n_{i=1} S_i =\sum^n_{i=1} \conv S_i$, where $\conv$ signifies the convex hull, and the sum over sets is to be understood as a Minkowski sum. Then,
\begin{itemize}
\item there is a point $x_i \in \conv S_i$ for each $i$ such that $x = \sum^n_{i=1} x_i$, and $x_i \in S_i$ except for at most $q$ values of $i$; and
\item there is a point $y_i \in  S_i$ for each $i$ such that $\|x-\sum_{i=1}^n y_i \|_{\mathbb{R}^q}\leq \sqrt{\min\{q,n\}} d$, where $d$ denotes the maximal diameter of $S_i$.
\end{itemize}
\end{lem}

\medskip

In the proofs of Lemmata \ref{lm1},  \ref{lm:analyseconvexification}, and \ref{lem:approxeq}, in order to simplify the notation, $i$ and  $x_{-i}\in \tX_{-i}$ are arbitrarily fixed. Index $i$ and the parameter $x_{-i}$ are thus omitted in $f_i$, $\tfi$, $\X_i$, $\tXi$ and $Z(i,\cdot)$. 

\begin{lem}\label{lm1} 
Under Assumption \ref{ass:lsc}, for each $x_{-i}\in  \tX_{-i}$, 
\begin{enumerate}[(1)]
\item $\tf_i (x_i, x_{-i}) \leq f_i (x_i, x_{-i}) $ for all $x_i\in \X_i$;

\item the infimum in~\eqref{eq:costConv} can be attained, i.e., it is in fact a minimum for all $x_i\in \X_i$;

\item the function $\tfi(\cdot, x_{-i})$ is l.s.c. and convex on $\tXi$, and $\conv\, (\epi\, f_i(\cdot,x_{-i}))= \epi\, \tfi(\cdot,x_{-i}) =  \convhull\, (\epi\, f_i(\cdot,x_{-i}))$; 

\item both $\tfi(\cdot,x_{-i})$ and $f_i(\cdot, x_{-i})$ attain their minima on $\tXi$ and $\X_i$ respectively; and
\begin{equation}
\label{eq:minmin}
\min_{\txi \in \tXi} \tfi (\txi, x_{-i})=\min_{x_i\in \X_i} f_i( x_i,x_{-i})\ .
\end{equation}
In particular, if $ \tx_i \in \arg\min_{y_i \in \tXi} \tfi (y_i, \tx_{-i})$, then  $Z(i,\tx) \subset \arg\min_{y_i\in \X_i} f_i( y_i, \tx_{-i})$, where $Z(i,\tx)$ is an arbitrary  generator for $(i,\tx)$ defined in Definition \ref{def:w_i}.
\end{enumerate}
\end{lem}

\begin{proof}[Proof of Lemma \ref{lm1}] ~~\\
The lemma is a particular case of more general results well-known in the field of convex analysis that have been shown in various works, such as \cite[Lemma X.1.5.3]{HiriartLemarechal1993b}. We will provide a proof for this particular case for the sake of completeness.

\noindent (1) For $x\in X$, in the definition of $\tf(x)$, take $x^{k}=x,\; \alpha^{k}=\frac{1}{d+1}$ for all $k$. By definition, $\tf(x)\leq \sum_{k=1}^{d+1}\alpha^{k}f(x^{k})= f(x)$.
\smallskip

\noindent (2) Suppose that  $((\alpha^{k,n})_{k},(x^{k,n})_{k})_{n\in\NN}$ is a minimizing sequence for $\tf(\tx)$, i.e., \linebreak[4]$\tf (\tx) = \lim_{n\rightarrow \infty}\sum_{k=1}^{d+1}\alpha^{k,n}f(x^{k,n})$, with $((\alpha^{k,n})_{k},(x^{k,n})_{k})_{n\in\NN}$ satisfying the conditions in \eqref{eq:costConv}. Since $(\alpha^{1,n})\in [0,1]$ for all $ n$,  it has a convergent subsequence $\alpha^{1,\phi_1(n)}$, which converges to some $\alpha^1$. Consider sequence $\alpha^{2,\phi_1(n)}$ which has a subsequence $\alpha^{2,\phi_2(n)}$ converging to some $\alpha^2$. Note that $\phi_2(n)$ is a subsequence of $\phi_1(n)$. Repeat this operation $d+1$ times and obtain the subsequences $\phi_{1}(n), \ldots, \phi_{d+1}(n)$ such that $\alpha^{k, \phi_k(n)}$ converges to $\alpha^k$, for $k=1,\ldots, d+1$. Consider $x^{1,\phi_{d+1}(n)}$, which is in the compact set $X$. It has a convergent subsequence $x^{1,\phi_{d+2}(n)}$ converging to $ x^1 \in \X$. Again, take a subsequence $\phi_{d+3}(n)$ such that $x^{2, \phi_{d+3}(n)}$ converges to $x^k$, and so on. Finally, one obtains a subsequence $\phi_{2d+2}(n)$ of $\NN$ such that 
\begin{align}
\tf (\tx) &= \lim_{n\rightarrow \infty}\sum_{k=1}^{d+1}\alpha^{k,\phi_{2d+2}(n)}f(x^{k,\phi_{2d+2}(n)})\, ,\label{1.1}\\
\alpha^{k} &= \lim_{n\rightarrow \infty} \alpha^{k,\phi_{2d+2}(n)}\, ,\;\alpha^{k}\in[0,1]\, , \; k=1,2,\cdots, d+1\, , \label{1.2}\\
\sum_{k=1}^{d+1}\alpha^{k}&=\lim_{n\rightarrow \infty}\sum_{k=1}^{d+1}\alpha^{k,\phi_{2d+2}(n)}=1\, ,\label{1.3}\\
x^{k} &= \lim_{n\rightarrow \infty} x^{k,\phi_{2d+2}(n)}\, ,\; x^{k}\in \X\, , \; k=1,2,\cdots, d+1\, ,\label{1.4}\\
\sum_{k=1}^{d+1}\alpha^{k}x^{k}&=\lim_{n\rightarrow \infty}\sum_{k=1}^{d+1}\alpha^{k,\phi_{2d+2}(n)}x^{k,\phi_{2d+2}(n)}= \lim_{n\rightarrow \infty}\tx =\tx\, .\label{1.5}
\end{align}
Then,
\begin{align*}
\sum_{k=1}^{d+1}\alpha^{k}f(x^{k})&\leq \varliminf_{n \rightarrow \infty}  \sum_{k=1}^{d+1}\alpha^{k}f(x^{k,\phi_{2d+2}(n)})  
= \varliminf_{n \rightarrow \infty}\sum_{k=1}^{d+1}\alpha^{k,\phi_{2d+2}(n)}f(x^{k,\phi_{2d+2}(n)})  \\
&= \tf (\tx)  
\leq \sum_{k=1}^{d+1}\alpha^{k}f(x^{k}) \ . 
\end{align*}
where the first inequality is due to \eqref{1.4}, the second equality is due to \eqref{1.2}, the third equality is due to \eqref{1.1} and the fourth inequality due to \eqref{1.3}, \eqref{1.5} and \eqref{eq:costConv}. This shows that $\tf(\tx)=\sum_{k=1}^{d+1}\alpha^{k}f(x^{k})$, i.e., $(\alpha^{k},x^{k})_{k=1}^{d+1}$, is a minimizer.
\smallskip

\noindent (3) 
On the one hand, for all $ (x, y)\in \conv\, (\epi\, f) $, by the Caratheodory theorem \cite[Proposition 1.2.1]{Bertsekas2009}, there exist $(x^{k}, y^{k})\in \epi\, f$, $k=1,\ldots, d+1$ such that $(x,y)=\sum_{k=1}^{d+1}\alpha^{k}(x^{k},y^{k})$, with $\alpha \in \mathcal{S}_{d}$. Hence, $y^{k}\geq f(x^{k})$, and $y=\sum_{k=1}^{d+1} \alpha^{k}y^{k}\geq \sum_{k=1}^{d+1}\alpha^{k}f(x^{k})\geq \tf(x)$. This shows that $(x,y)\in  \epi\,\tf$. Therefore, $\conv\, (\epi\, f)\subset \epi\, \tf$. Recall that $f$ is l.s.c.; hence, $\epi\, f$ is a closed set and thus so is $\conv\, (\epi\, f)$. Thus, $\convhull\, (\epi\, f)\subset \epi\, \tf$.

On the other hand, for all $ (x,y)\in \epi\,\tf$, $y\geq \tf(x)$. Let $((\alpha^{k,n})_{k},(x^{k,n})_{k})_{n\in\NN}$ be the minimizing sequence for $\tf(x)$, i.e., $\tf (x) = \lim_{n\rightarrow \infty}\sum_{k=1}^{d+1}\alpha^{k,n}f(x^{k,n})$, with $\alpha^{k,n},x^{k,n}$ satisfying the conditions in \eqref{eq:costConv}. Then, $y = \lim_{n\rightarrow \infty}\sum_{k=1}^{d+1}\alpha^{k,n}(f(x^{k,n})+ \frac{\delta}{d+1})$, where $\delta=y-\tf(x)\geq 0$. Denote $y^{n}=\sum_{k=1}^{d+1}\alpha^{k,n}(f(x^{k,n})+ \frac{\delta}{d+1})$. Then, $(x,y^{n})\in \conv (\epi\, f)$, and $\lim_{n\rightarrow \infty}(x,y^{n})=(x,y)$. This means that  $(x,y)\in \convhull(\epi\,f)$ and, therefore, $ \epi\, \tf \subset  \convhull\, (\epi\, f)$.

In conclusion, $\epi\, \tf(\cdot) = \convhull\, (\epi\, f(\cdot))$, which implies that the epigraph of $\tf$ is closed and convex. Thus, $\tf$ is l.s.c. and convex on $\tX$.
\smallskip

\noindent (4) By the lower semicontinuity of $\tf$ and $f$ on the compact sets $\tX$ and $X$, their minima can be attained. The equality \eqref{eq:minmin} is thus clear by the definition in \eqref{eq:costConv}.
\end{proof}

\begin{rem}
 If $f_i(\cdot,x_{-i})$ is not l.s.c, the inclusion relationship in Lemma \ref{lm1}(2) can be strict, as shown, respectively, by the following two examples of dimension 1.

\begin{itemize}
	\item $\X=\{0\}\cup\{\pm\frac{1}{z}\}_{z\in \NN^{*}}$, $f(x)=|x|$ for $x\in \X\setminus \{0\}$, and $f(0)=1$. Then, $\tf(x)=|x|$, for all $x \in \tX = [-1,1]$, and $\conv\, (\epi\, f)\subsetneq \epi\, \tf$.
	
	\item $\X=[0,1]$, $f(x)=0$ for $x \neq 0$, and $f(0)=1$. Then, $\tf(x)=f(x)$ for all $x\in[0,1]$, and $\epi\, \tf\subsetneq \convhull\, (\epi\, f)$.
\end{itemize}
\end{rem}
\medskip

\begin{lem}\label{lm:analyseconvexification}
Under Assumption \ref{ass:lsc}, for any profile $\tx\in \tX$, for any player $i$, 
 for all $x_i\in Z(i,\tx)$,
\begin{enumerate}[(1)]
\item $f_i(x_i, \tx_{-i})=\tfi(x_i, \tx_{-i})$;
\item for any $h\in \partial_i \tfi(\txi, \tx_{-i})$,
\begin{equation}\label{eq:h}
f_i(x_i, \tx_{-i})=\tfi(x_i, \tx_{-i}) = \tfi(\txi, \tx_{-i}) + \langle h, x_i - \txi \rangle\, .
\end{equation}
\end{enumerate}
\end{lem}

\begin{proof}[Proof of Lemma \ref{lm:analyseconvexification}] 
Let $\{x^1, \ldots, x^{d+1}\} \subset \X$ be a generator of $(\tx, \tf(\tx))$ and let $\alpha\in \S_d$ be their corresponding weights. 
\smallskip

\noindent (1) Suppose that there is a $k$ such that $f(x^k)>\tf(x^k)$. Then, there exists $(y^l)_l$ in $\X$ and $\beta\in \S_d$ such that $x^k=\sum_l \beta^l y^l$ and $\tf(x^k)=\sum_l \beta^l f(y^l) < f(x^k)$. In consequence, $\tf(\tx)=\sum_m \alpha^m f(x^m) > \sum_{m \neq k}\alpha^m f(x^m) + \sum_{l} \alpha^k \beta^l f(y^l)$, while $\sum_{m \neq k}\alpha^m x^m + \sum_{l} \alpha^k \beta^l y^l=\tx$ and $\sum_{m \neq k}\alpha^m + \sum_{l} \alpha^k \beta^l=1$, contradicting the definition of $\tf(\tx)$.
\smallskip

\noindent (2) By the definition of the subdifferential, one has
\begin{equation}\label{eq:subgrad}
\tf(x^k)  \geq  \tf(\tx) + \langle h, x^k - \tx \rangle \ , \quad \forall k =1, \ldots, d+1 \ .
\end{equation}
Multiplying \eqref{eq:subgrad} by $\alpha^k$ for each $k$ and adding the $d+1$ inequalities yield
\begin{equation}\label{eq:repeat}
\sum_{k=1}^{d+1} \tf(x^k) \geq \tf(\tx) + \langle h, \sum_k \alpha^k x^k - \tx \rangle \quad \Leftrightarrow \quad \tf(\tx) \geq \tf(\tx) \ .
\end{equation}
If, for at least one $k$, the inequality in \eqref{eq:subgrad} is strict, then the inequalities in \eqref{eq:repeat} are strict as well, which is absurd. Therefore, for each $k$, $\tf(x^k) = \tf(\tx) + \langle h, x^k - \tx \rangle$.
\end{proof}

\begin{proof}[Proof of Lemma \ref{lem:approxeq}] 
First note that $\tx$ is in $\ri( \conv \, Z(\tx))$, the relative interior of $\conv\, Z(\tx)$. Hence, for $t>0$ small enough, $\tx\pm t(x - \tx)$ is in $\ri( \conv\, Z(\tx))\subset \tX$. By \eqref{eq:firstorder}, $\bigl\langle  h, \tx\pm t(x - \tx) - \tx \bigr\rangle  \geq - \eta \|\tx \pm t(x - \tx) -\tx \|$, which yields $\big|\bigl\langle  h, x - \tx  \bigr\rangle \big| \leq \eta \|x - \tx\|$.  Then, by Lemma \ref{lm:analyseconvexification}, $|\tf(x) - \tf(\tx)| =  |\langle h, x - \bx  \rangle| \leq \eta\|x -\tx \|$.
\end{proof}

\begin{proof}[Proof of Theorem \ref{thm:main}]
For each $i\in N$, define a set 
$E_i(\tx) :=  A_i Z(i,\tx)$ in $\RR^q$.
Since $\txi \in \conv\, (Z(i,\tx))$, one has $\sum_{i\in N} A_i \tx_i \in \sum_{i\in N}\conv (E_i(\tx))= \conv \Big(\sum_{i\in N} E_i(\tx)\Big)$ by the linearity of the $A_i$'s. 
According to the Shapley-Folkman Lemma, there exists $e_i \in \conv(E_i(\tx))$ for each $i\in N$, and a subset $I\subset N$ with $\vert I\vert \leq q$, such that (i) $\sum_{i\in N}A_i \tx_i =\sum_{i\in N} e_i$, and (ii) $e_i\in E_i(\tx)$ for all $i\notin I$. Thus, for all $i\notin I$, there exists $\bxi\in Z(i,\tx)$, such that $e_i = A_i \bxi$. 
For all $i\in I$, take arbitrarily $\bxi \in Z(i,\tx)$. Then,
\begin{align}
\Big\|\sum_{i\in N} A_i \tx_i-\sum_{i\in N} A_i x^{*}_i\Big \|& \leq \Big\|\sum_{i\in N} A_i \tx_i-\sum_{i\in N} A_i \bxi \Big \| =  \Big\|\sum_{i\in N}e_i-\sum_{i\in N} A_i \bxi \Big \|=  \Big\|\sum_{i\in I} A_i (\tx_i-\bxi) \Big \| \notag\\
&  \leq \sqrt{q} M \Delta \ . \label{eq:SF ineq}
\end{align}


Now, for all $i$, $x_i^*\in  Z(i,\tx)$, so that it satisfies 
\begin{equation}
\label{eq:xstarmin}
f_i(x_i^*,\tx_{-i}) \leq \tfi(\txi,\tx_{-i}) \leq  \tfi(x_i,\tx_{-i}) + \epsilon+\eta \leq f_i(x_i,\tx_{-i}) + \epsilon+\eta \, ,\quad\textrm{for all}\ x_i\in \X_i\, ,
\end{equation}
according to Lemma \ref{lm:analyseconvexification}.(1), Lemma \ref{lem:approxeq} and Lemma \ref{lm1}.(1).

Recall that $f_i(x)=\theta_i (x_i,\, \frac{1}{n}\sum_{j\in N} A_j \, x_j )$. Hence, for any $x_i\in \X_i$ 
\begin{align*}
f_i(x_i,\, \tilde x_{-i}) 
& =  \theta_i \bigg(x_i,\, \frac{1}{n} A_i\, x_i+ \frac{1}{n}\sum_{j\in N_{-i}} A_j \, \tx_j \bigg) 
\\
& = \theta_i \bigg(x_i, \, \frac{1}{n} A_i\, x_i+ \frac{1}{n}\sum_{j\in N_{-i}} A_j\,  x^*_j+\frac{1}{n} A_i\, (x^*_i-\tx_i)+\frac{1}{n}\sum_{j\in N}A_j\, (\tx_j-x^*_j) \bigg)\\
& = \theta_i \bigg(x_i, \,\frac{1}{n} A_i \, x_i + \frac{1}{n}\sum_{j\in N_{-i}} A_j \, x^*_j+\delta_i \bigg) - \theta_i \bigg(x_i, \,\frac{1}{n} A_i \, x_i + \frac{1}{n}\sum_{j\in N_{-i}} A_j \, x^*_j \bigg) + f_i(x_i, x^*_{-i})\ ,
\end{align*}
where $\delta_i:= \frac{1}{n} A_i\, (x_i^*-\tilde x_i)+\frac{1}{n}\sum_{j\in N}A_j\,(\tx_j-x^*_j)$. 

By \eqref{eq:SF ineq}, $\Vert \delta_i\Vert \leq \frac{(\sqrt{q}+1) M \Delta}{n}$. Using Assumption~\ref{ass:lsc}, we can show that, for any $x_i\in \X_i$,
\begin{equation*}
\left\vert f_i(x_i,x^*_{-i})-f_i(x_i,\tilde x_{-i}) \right\vert \leq H \left(\frac{(\sqrt{q}+1)M\Delta}{n}\right)^\gamma \ .
\end{equation*}
Injecting this result into~\eqref{eq:xstarmin} yields 
\begin{equation}\label{eq:error}
f_i(x_i^*, x^*_{-i})\leq  f_i(x_i,x^*_{-i})+ \epsilon+\eta + 2 H \left(\frac{(\sqrt{q}+1)M\Delta}{n}\right)^\gamma \, ,\quad \forall x_i\in \X_i\, ,\; \forall i\in N\, .
\end{equation} 
\end{proof}

\begin{lem}\label{lm:decomposeSF}
Under Assumption \ref{ass:lsc}, suppose that $\tx\in \tX$ is an $\epsilon$-PNE in $\tGamma$ satisfying the $\eta$-stability condition with respect to $(Z(i,\tx))_i$, where $Z(i,\tx) = \{x_i^1, x_i^2,\ldots, x_i^{l_i} \}$ with $1\leq l_i\leq d+1$ and $\tx_i = \sum_{l=1}^{l_i} \alpha_i^l x_i^l$, where $\alpha \in \S_{l_i-1}$. Each player $i$ plays a mixed strategy independently, i.e., a random action $X_i$ following the distribution $\tilde{\mu}_i$ over $\X_i$ defined by $\mathbb{P}(X_i = x_i^{l}) = \alpha_i^l$. In other words,
\begin{equation}\label{eq:tildemu}
\tilde{\mu}_i = \sum _{l=1}^{l_i} \alpha_i^l \delta_{x_i^l}\ ,
\end{equation}
where $\delta_{x_i^l}$ stands for the Dirac distribution on $x_i^l$. Then,
\begin{equation*}
\mathbb{E}\Big\|\sum_{i\in N} A_i \tx_i-\sum_{i\in N} A_i X_i\Big \| \leq  \sqrt{n} M\Delta\, .
\end{equation*}
\end{lem}
\begin{proof}
By the independence of $X_i$, the $A_i X_i$ are independent of each other. From the definition of $\tilde{\mu}_i $,  $\mathbb{E}( A_i X_i ) = A_i \tx_i$. Therefore, 
\begin{equation*}
\left(\mathbb{E}\Big \|\sum_{i\in N} A_i \tx_i-\sum_{i\in N} A_i X_i\Big \|\right)^2
\leq	 \mathbb{E}\left[ \Big\|\sum_{i\in N} A_i \tx_i-\sum_{i\in N} A_i X_i\Big \|^2 \right]
= \sum_{i\in N} \mathbb{V}\mathrm{ar}(A_i X_i)\leq n M^2 \Delta^2\, ,
\end{equation*}
where the first inequality is due to Jensen's inequality.
\end{proof}

\begin{proof}[Proof of Proposition \ref{prop:mixed}]
By the same arguments used in the proof of Theorem \ref{thm:main}, one has
\begin{equation*}
| f_i(x_i,\, \tx_{-i}) - f_i(x_i, X_{-j}) |\leq H\, \|\delta_i(X)\|^{\gamma}\, ,
\end{equation*}
where $\delta_i(X):= \frac{1}{n} A_i\, (X_i - \tx_i)+\frac{1}{n}\sum_{j\in N}A_j\, (\tx_j - X_j)$. By Lemma \ref{lm:decomposeSF},
\begin{equation*}
\mathbb{E} \|\delta_i(X)\| \leq \frac{1+\sqrt{n}}{n} M \Delta\ .
\end{equation*}

Besides, since $X_i$ takes values in $Z(i,\tx)$,
\begin{align*}
f_i(X_i, X_{-i}) & = f_i(X_i, X_{-i}) - f_i(X_i, \tx_{-i}) + f_i(X_i, \tx_{-i}) \\
& \leq f_i(X_i, X_{-i}) - f_i(X_i, \tx_{-i}) + f_i(x_i, \tx_{-i}) + \epsilon + \eta \\
& = f_i(X_i, X_{-i}) - f_i(X_i, \tx_{-i}) + f_i(x_i, \tx_{-i})- f_i(x_i, X_{-i})+ f_i(x_i, X_{-i})  + \epsilon + \eta \, ,
\end{align*}
so that 
\begin{align*}
 f_i(X_i, X_{-i}) -f_i(x_i, X_{-i}) & \leq  |f_i(X_i, X_{-i}) - f_i(X_i, \tx_{-i})|+ |f_i(x_i, \tx_{-i})- f_i(x_i, X_{-i})|+ \epsilon+\eta \\
& \leq 2 H(\delta_i(X))^{\gamma} + \epsilon + \eta\, .
\end{align*} 

Therefore,
\begin{equation*}
\mathbb{E}\big[ f_i(X_i, X_{-i}) - f_i(x_i,X_{-i}) \big] \leq 2 H \,\mathbb{E}\big[\|\delta_i(X)\|^{\gamma}\big] + \epsilon + \eta \leq  2 H \left(\frac{(\sqrt{n}+1)M\Delta}{n}\right)^\gamma + \epsilon+\eta\ .
\end{equation*}
\end{proof}

\begin{proof}[Proof of Lemma \ref{lm: approx game}] 
(1) First show that, for any fixed $x_i\in \X_i$, the function $\theta_i(x_i, y) := \big\langle g (y+\frac{a_i}{n} (x^0_i-x_i)), x_i \big\rangle  + \ell_i(x_i)$ is $L_g\Delta$-Lipschitz in $y$ on $\Omega$. For this, fix $x_i\in \X_i$. For any $y$ and $y'$ in $\Omega$,
 	\begin{equation*}
 	\begin{split}
 	\vert \theta_i(x_i,y')-\theta_i(x_i,y)\vert ^2
    &=\Big\vert \Big\langle g \Bigl( y'+\frac{a_i}{n} (x^0_i-x_i)\Bigr) -  g \Bigl( y+\frac{a_i}{n} (x^0_i-x_i)\Bigr), x_i \Big\rangle\Big\vert ^2\\
 	&\leq \Big \| g \Bigl( y'+\frac{a_i}{n} (x^0_i-x_i)\Bigr) -  g \Bigl( y +\frac{a_i}{n} (x^0_i-x_i)\Bigr) \Big\|^2 \Delta^2\\
 	& = \sum_{t=1}^{d}\Bigl( g_t \Bigl( y'_t +\frac{a_i}{n} (x^0_{i,t}-x_{i,t})\Bigr) -  g_t \Bigl( y_t +\frac{a_i}{n} (x^0_{i,t}-x_{i,t})\Bigr) \Bigr)^2 \Delta^2\\
 	&\leq L_g^2 \Delta^2 \sum_{t=1}^{d}(y'_t-y_t)^2 \\
 	&= L_g^2 \Delta^2 \|y'-y\|^2\, ,
 	\end{split}
 	\end{equation*}
 	where the first inequality results from the Cauchy-Schwarz inequality, while the second inequality is true because $g_t$ is $L_{g_t}$-Lipschitz.

(2) It is easy to see that $|f_i(x_i , x_{-i}) - h_i(\frac{1}{n} \sum_{j\in N} a_j  x_j) - \bfi(x_i , x_{-i})| \leq \frac{L_g M\Delta^2}{n}$ for all $x_i\in \X_i$ and all $x_{-i}\in \tX_{-i}$. Hence, if $\bx\in X$ is an $\epsilon$-PNE of $\bGamma$, then, for each $i$, for any $x_i\in \X_i$, 
\begin{equation*}
	\begin{split}
	f_i(\bxi, \bx_{-i})\leq& \bfi(\bxi, \bx_{-i}) + h_i \Big(\frac{1}{n} \sum_{j\in N} a_j  \bx_j \Big)+ \frac{L_g M\Delta^2}{n} \\
	\leq &\bfi(x_i, \bx_{-i})+\epsilon +  h_i \Big(\frac{1}{n} a_i x_i + \frac{1}{n} \sum_{j\neq i}^n a_j  \bx_j \Big) + \frac{L_h M \Delta}{n} +\frac{L_g M\Delta^2}{n}\\
	 \leq &f_i(x_i, \bx_{-i})+\epsilon + \frac{L_h M \Delta}{n}+\frac{2 L_g M\Delta^2}{n}\, ,
	\end{split}
\end{equation*}
where the second inequality is due to the definition of $\epsilon$-PNE and the Lipschitz continuity of $h_i$.
\end{proof}

\begin{lem} \label{lm:convergence}
Under Assumption \ref{ass:example}, let $(x^{k})_{k\in \NN}$ be the sequence generated by Algorithm \ref{algo:BCP} with some initial point $x^0\in \tX$. Then,
\begin{enumerate}[(1)]
\item $\sum_{k=1}^{\infty}\|x^{k-1}-x^{k}\|^2\leq \frac{2n^2}{m^2 L_g} C $, where $C=(d\Delta L_g + 2 B_r) M$;
\item for any $K \in \NN^*$,  there exists some $k^{*}\leq K$, such that $\|x^{k^{*}-1}-x^{k^{*}}\| \leq \frac{\sqrt{2C} \, n }{m \sqrt{L_g K}}$.
\end{enumerate}
\end{lem}
\begin{proof}[Proof of Lemma \ref{lm:convergence}]
Consider the following two real-valued functions defined on $\tX$:
\begin{equation}\label{eq:sanpotential}
G_0(x) := \sum_{t=1}^{d} G_{t} \Bigl(\frac{1}{n} \sum_{j\in N} a_j x_{j,t} \Bigr) \, , \quad G(x) :=  G_0(x) + \sum_{j\in N} \frac{a_j}{n} \tr_j(x_j) \ ,
\end{equation}
where $G_{t}$ is a primitive function of $g_t$, which exists thanks to Assumption \ref{ass:example}.

Note that the function $G_0$ is convex and differentiable on a neighborhood of $\tX$, and the convex function $\tr_j$ is uniformly bounded on $\tX_j$ for all $j\in N$ with the same bound $B_\ell$, according to Assumption \ref{ass:example}. 

Besides, it is easy to see that, for any $i$ and fixed $x_{-i}\in \tX_{-i}$, $\nabla_i G_0(x_i,x_{-i}):=\frac{\partial G_0(x_i, x_{-i})}{\partial x_i}=\frac{a_i}{n}g (\frac{1}{n}a_i x_i + \frac{1}{n}\sum_{j\neq i} a_j x_j)$ is $\frac{a_i^2L_g}{n^2}$-Lipschitz continuous on $\tXi$. 

Therefore, Assumptions 1 and 2 in \cite{XuandYin2013} are verified. One can thus apply Lemma 2.2 from \cite{XuandYin2013} and obtain
\begin{equation*}
\sum_{i\in N} \frac{a_i^2 L_g}{2n^2}\|x^{k}_i-x^{k+1}_i\|^2\leq  G(x^{k})-G(x^{k+1})\, ,
\end{equation*}
so that
\begin{equation*}
\|x^{k}-x^{k+1}\|^2 \leq \frac{2n^2}{m^2L_g}(G(x^{k})-G(x^{k+1}))\, .
\end{equation*} 
In consequence,
\begin{equation}\label{eq:convergence}
\sum_{k=0}^{\infty}\|x^{k}-x^{k+1}\|^2 \leq \frac{2n^2}{m^2L_g}(G(x^{0})-G_{min})\, ,
\end{equation}
where $G_{min}$, defined as $\inf_{\{x\in \tX\}}G(x)$, exists and is finite, because $G$ is l.s.c. on the compact set $\tX$. Suppose that $G_{min}$ is attained at $\underline{x}\in \tX$; then,
\begin{equation}\label{eq:C}
\begin{split}
G(x^0)-G_{min} &=G(x^0)-G(\underline{x})\\
& =\sum_{t=1}^{d}\int_{\frac{1}{n} \sum_{j\in N} a_j \underline{x}_{j,t} }^{\frac{1}{n} \sum_{j\in N} a_j x^0_{j,t}}g_{t}(s)ds + \sum_{j\in N} \frac{a_j}{n} (\tr_j(x^0_j)- \tr_j(\underline{x}_j))\\
&\leq d M\Delta B_g + 2 M B\ ,
\end{split}
\end{equation}
where	the last inequality is due to the mean value theorem and Assumption \ref{ass:example}.
Combining \eqref{eq:convergence} and \eqref{eq:C} yields $\sum_{k=0}^{\infty}\|x^{k}-x^{k+1}\|^2 \leq \frac{2n^2}{m^2L_g} C$. 
This immediately implies
    \begin{equation*}
    	\sum_{k=1}^{K} \|x^{k-1}-x^{k}\|^2\leq \frac{2n^2}{m^2 L_g} C \, .
    \end{equation*}
    The second result of the lemma is then straightforward.
\end{proof}
\bigskip

\begin{proof}[Proof of Proposition \ref{lm:Nash-disagrregative}]
First, notice that the vector function $\zeta : \tX \rightarrow \RR^{d}, x \mapsto \zeta(x) = g\bigl( \frac{1}{n}  \sum_{j\in N} a_j  x_j\bigr)$  is $\frac{L_g M}{\sqrt{n}}$-Lipschitz continuous, i.e., $\|\zeta(x)-\zeta(y)\|\leq \frac{L_g M}{\sqrt{n}}\|x-y\|$, for all $x,y\in \tX$. Indeed, $
\|\zeta(x)-\zeta(y)\|^2 = \sum_{t=1}^{d} |g_{t}( \frac{1}{n} \sum_{j\in N} a_j x_{j,t} )-g_{t} ( \frac{1}{n} \sum_{j\in N} a_j y_{j,t} ) |^2 \leq \linebreak[4] \sum_{t=1}^{d}\big| \frac{L_{g_t}}{n}|\sum_{j\in N}a_j(x_{j,t}-y_{j,t})|\big|^2 \leq \sum_{t=1}^{d} \big( \frac{L_g^2}{n^2}\sum_{j=1}^{n}a_j^2 \sum_{j\in N}(x_{j,t}-y_{j,t})^2\big)\leq \frac{L_g^2 M^2}{n}\|x-y\|^2$, 
where the first inequality is because $g_t$ is $L_{g_t}$-Lipschitz, while the second results from the Cauchy-Schwarz inequality.
   
Next, suppose that the sequence $(x^{k})_{k\in \NN}$ is generated by Algorithm \ref{algo:BCP} with some initial point $x^0\in \tX$. Let us show that, if $\|x^{k-1}-x^{k}\|\leq u_k$, then, $x^{k}$ satisfies the full $\eta(u_k)\Delta$-stability condition and, furthermore, it is an $\eta(u_k)\Delta$-PNE of game $\tGamma$, where $\eta(u_k) = \frac{L_gM u_k}{\sqrt{n}}+\frac{ 2  L_g M \Delta}{n}$.

Since $\|x^{k}-x^{k-1}\|\leq u_k$, one has 
	$\|  (x_{1}^{k}, \ldots, x_{i-1}^{k}, x_i^{k} , x_{i+1}^{k}, \ldots, x_{n}^{k} )- (x_{1}^{k}, \ldots, x_{i-1}^{k}, x_{i}^{k-1} , \linebreak[4]x_{i+1}^{k-1}, \ldots, x_{n}^{k-1} )  \|\leq u_k$.
	Thus, the Lipschitz continuity of $\zeta$ on $\tX$ and the Lipschitz continuity of $g$ in $x_i$ imply that
\begin{equation}\label{eq:lip inequality}
\begin{split}
&~~\Big \|g \Bigl(  \frac{1}{n} \sum_{j \neq i} a_j  x^{k}_j + \frac{1}{n} a_i x^0_i \Bigr) -   g \Bigl(  \frac{1}{n} \sum_{j < i} a_j  x^{k}_j + \frac{1}{n} a_i x^{k-1}_i + \frac{1}{n}\sum_{j > i} a_j  x^{k-1}_j \Bigr) \Big\| \\
&\leq \Big \|g \Bigl(  \frac{1}{n} \sum_{j \neq i} a_j  x^{k}_j + \frac{1}{n} a_i x^0_i \Bigr) -   g \Bigl(  \frac{1}{n} \sum_{j\in N} a_j  x^{k}_j \Bigr) \Big\| \\
&~~ + \Big \| g \Bigl(  \frac{1}{n} \sum_{j\in N} a_j  x^{k}_j \Bigr) -   g \Bigl(  \frac{1}{n} \sum_{j < i} a_j  x^{k}_j + \frac{1}{n} a_i x^{k-1}_i + \frac{1}{n}\sum_{j > i} a_j  x^{k-1}_j \Bigr) \Big\|\\
&\leq \frac{L_g M\Delta}{n}+\frac{L_g M u_k}{\sqrt{n}}\ .
\end{split}
\end{equation}
	
The first order condition of optimality of the optimization problem \eqref{alg:proximal1} is the following:  there exists some $p_i$ in the subdifferential of $\tri(x^k_i)$ at $x^k_i$, denoted by $\partial \tri (x^k_i)$, such that for all $x_i\in \tX_i$,
\begin{equation}\label{eq:first order inequality}
\Big\langle  g \Bigl(  \frac{1}{n} \sum_{j < i} a_j  x^{k}_j + \frac{1}{n}\sum_{j \geq i} a_j  x^{k-1}_j \Bigr)  + \frac{a_iL_g}{n} ( x^{k}_i-x_{i}^{k-1} ) + p_i , x_i - x^{k}_i \Big \rangle \geq 0 \, .
\end{equation}
Then, 
	\begin{equation*}
	\begin{split}
	&  \Big\langle g \Bigl(  \frac{1}{n} \sum_{j \neq i} a_j  x^{k}_j + \frac{1}{n} a_i x^+_i \Bigr)  + p_i, x_i - x^{k}_i\Big\rangle \\
	& = \Big\langle g \Bigl(  \frac{1}{n} \sum_{j \neq i} a_j  x^{k}_j + \frac{1}{n} a_i x^+_i \Bigr) - g \Bigl(  \frac{1}{n} \sum_{j < i} a_j  x^{k}_j + \frac{1}{n}\sum_{j \geq i} a_j  x^{k-1}_j \Bigr)  , x_i - x^{k}_i\Big\rangle\\
	&~~ +\Big\langle  g \Bigl(  \frac{1}{n} \sum_{j < i} a_j  x^{k}_j + \frac{1}{n}\sum_{j \geq i} a_j  x^{k-1}_j \Bigr) + \frac{a_iL_g}{n}( x^{k}_i-x_{i}^{k-1} ) + p_i , x_i - x^{k}_i \Big\rangle\\
	&~~ - \Big\langle \frac{a_iL_g}{n}( x^{k}_i-x_{i}^{k-1} ) , x_i - x^{k}_i\Big\rangle\\
	&\geq \Big\langle g \Bigl(  \frac{1}{n} \sum_{j \neq i} a_j  x^{k}_j + \frac{1}{n} a_i x^+_i \Bigr) -   g \Bigl(  \frac{1}{n} \sum_{j < i} a_j  x^{k}_j + \frac{1}{n} a_i x^{k}_i + \frac{1}{n}\sum_{j > i} a_j  x^{k-1}_j \Bigr) , x_i - x^{k}_i\Big\rangle\\
	&~~ - \Big\langle \frac{a_iL_g}{n}( x^{k}_i-x_{i}^{k-1} ) , x_i - x^{k}_i\Big\rangle\\
	&\geq - \Big(\frac{L_gM u_k}{\sqrt{n}}+\frac{L_g M\Delta}{n}+\frac{a_i L_g \Delta}{n} \Big)\|x_i-x^{k}_i\|\\
	&\geq -\Big(\frac{L_gM u_k}{\sqrt{n}}+\frac{2 L_g M \Delta}{n}\Big)\|x_i-x^{k}_i\| = -\eta(u_k)\|x_i-x^{k}_i\|  \ ,
	\end{split}
	\end{equation*}
where the first inequality is due to \eqref{eq:first order inequality}, while the second inequality is due to \eqref{eq:lip inequality} and the Cauchy-Schwarz inequality. 
	Then, according to Lemma \ref{lem:approxeq}, $x^{k}$  satisfies the full $\eta(u_k)\Delta$-stability condition for game $\tGamma$, where $\eta(u_k) = \frac{L_gM u_k}{\sqrt{n}}+\frac{2 L_g M \Delta}{n}$.
	
Furthermore, since $\tri$ is convex on $\tX_i$,
	\begin{equation*}
	\begin{split}
	\bar{f}_i(x_i,x^{k}_{-i})-\bar{f}_i(x^{k})&=\Big\langle g\Bigl(\frac{1}{n} \sum_{j\neq i} a_j x_j^{k}+\frac{1}{n} a_ix_i^+\Bigr), x_{i}-x^{k}_{i}\Big\rangle +\tri(x_i)- \tri(x_i^{k})\\
	&\geq \Big\langle g\Bigl(\frac{1}{n} \sum_{j\neq i} a_j x_j^{k}+\frac{1}{n} a_ix_i^+\Bigr), x_{i}-x^{k}_{i}\Big\rangle + \langle p_i, x_i-x^{k}_i\rangle \\
	&\geq -(\frac{L_gM u_k}{\sqrt{n}}+\frac{2 L_g M \Delta}{n})\|x_i-x^{k}_i\|\, .
	\end{split}
	\end{equation*}	
	Thus, $x^k$ is an  $\eta(u_k)\Delta$-PNE of game $\tGamma$.
\medskip

For any $K \in \NN^*$, there exists some $k^{*}\leq K$ such that $\|x^{k^{*}-1}-x^{k^{*}}\| \leq \frac{\sqrt{2C} n }{m \sqrt{L_g K}}$ according to Lemma \ref{lm:convergence}(2). The conclusion is immediately obtained by taking $\omega(K,n)=\eta\big(\frac{\sqrt{2C} n }{m \sqrt{L_g K}}\big)$.
\end{proof}

\begin{proof}[Proof of Theorem \ref{thm:congestionPNE}]
Proposition \ref{lm:Nash-disagrregative} shows that $x^{k^*}$ is an approximate PNE of game $\tGamma$ (obtained through the convexification of the non-convex auxiliary game $\bGamma$). Then, Theorem \ref{thm:main} is applied to show that (the ``Shapley-Folkman disaggregation'' of $x^{k^*}$) $x^*$ is an approximate PNE of the non-convex auxiliary game $\bGamma$. The use of Theorem \ref{thm:main}  is justified by Lemma \ref{lm: approx game}(1). Finally, Lemma \ref{lm: approx game}(2) is evoked to show that $x^*$ is an approximate PNE of the original non-convex game $\Gamma$.
\end{proof}


\bibliographystyle{amsplain}
\bibliography{SFGames} 

\providecommand{\bysame}{\leavevmode\hbox to3em{\hrulefill}\thinspace}
\providecommand{\MR}{\relax\ifhmode\unskip\space\fi MR }
\providecommand{\MRhref}[2]{%
  \href{http://www.ams.org/mathscinet-getitem?mr=#1}{#2}
}
\providecommand{\href}[2]{#2}
\begin{thebibliography}{10}

\bibitem{AubinEkeland1976}
J.P. Aubin and I.~Ekeland, \emph{Estimates of the duality gap in nonconvex
  optimization}, Mathematics of Operations Research \textbf{1} (1976), no.~3,
  225--245.

\bibitem{BasileAl2016}
A.~Basile, M.G. Graziano, and M.~Pesce, \emph{Oligopoly and cost sharing in
  economics with public goods}, International Economic Review \textbf{57}
  (2016), no.~2, 487--505.

\bibitem{BertsekasAl1983}
D.~{Bertsekas}, G.~{Lauer}, N.~{Sandell}, and T.~{Posbergh}, \emph{Optimal
  short-term scheduling of large-scale power systems}, IEEE Transactions on
  Automatic Control \textbf{28} (1983), no.~1, 1--11.

\bibitem{Bertsekas1979}
D.P. Bertsekas, \emph{Convexification procedures and decomposition methods for
  nonconvex optimization problems}, Journal of Optimization Theory and
  Applications \textbf{29} (1979), 169--197.

\bibitem{Bertsekas1996}
\bysame, \emph{{Constrained-Optimization and Lagrangian Multiplier Methods}},
  Athena Scientific, 1996.

\bibitem{Bertsekas2009}
\bysame, \emph{{Convex Optimization Theory}}, Athena Scientific, 2009.

\bibitem{BertsekasSandell1982}
D.P. {Bertsekas} and N.R. {Sandell}, \emph{Estimates of the duality gap for
  large-scale separable nonconvex optimization problems}, 21st IEEE Conference
  on Decision and Control, 1982, pp.~782--785.

\bibitem{BiTang2020}
Y.~Bi and A.~Tang, \emph{Duality gap estimation via a refined
  {Shapley--Folkman} lemma}, SIAM Journal on Optimization \textbf{30} (2020),
  no.~2, 1094--1118.

\bibitem{Corchon1994}
L.C. Corch\'on, \emph{Comparative statics for aggregative games the strong
  concavity case}, Mathematical Social Sciences \textbf{28} (1994), no.~3,
  151--165.

\bibitem{Dafermos1980}
S.~Dafermos, \emph{Traffic equilibrium and variational inequalities},
  Transportation Science \textbf{14} (1980), no.~1, 42--54.

\bibitem{PerezDavidVerdier1992}
J.~David, P.~Castrillo, and T.~Verdier, \emph{A general analysis of
  rent-seeking games}, Public Choice \textbf{73} (1992), no.~3, 335--350.

\bibitem{EkelandTemam1999}
I.~Ekeland and R.~T\'emam, \emph{{Convex Analysis and Variational Problems}},
  Society for Industrial and Applied Mathematics, 1999.

\bibitem{FacchineiPang2003}
F.~Facchinei and J.~Pang, \emph{{Finite-Dimensional Variational Inequalities
  and Complementarity Problems}}, Springer-Verlag New York, 2003.

\bibitem{FangLiuWang2019}
E.X. Fang, H.~Liu, and M.~Wang, \emph{Blessing of massive scale: spatial
  graphical model estimation with a total cardinality constraint approach},
  Mathematical Programming \textbf{176} (2019), no.~1--2, 175--205.

\bibitem{FoucartWan2018}
R.~Foucart and C.~Wan, \emph{Strategic decentralization and the provision of
  global public goods}, Journal of Environmental Economics and Management
  \textbf{92} (2018), 537--558.

\bibitem{HiriartLemarechal1993b}
J.B. Hiriart-Urruty and C.~Lemarechal, \emph{{Convex Analysis and Minimization
  Algorithms II: Advanced Theory and Bundle Methods}}, Springer-Verlag Berlin
  Heidelberg, 1993.

\bibitem{HofbauerSigmund1998}
J.~Hofbauer and K.~Sigmund, \emph{{Evolutionary Games and Population
  Dynamics}}, Cambridge University Press, 1998.

\bibitem{Horta2018}
J.~{Horta}, E.~{Altman}, M.~{Caujolle}, D.~{Kofman}, and D.~{Menga},
  \emph{Real-time enforcement of local energy market transactions respecting
  distribution grid constraints}, 2018 IEEE International Conference on
  Communications, Control, and Computing Technologies for Smart Grids
  (SmartGridComm), 2018, pp.~1--7.

\bibitem{HreinssonAl2021}
K.~Hreinsson, A.~Scaglione, M.~Alizadeh, and Y.~Chen, \emph{New insights from
  the {Shapley-Folkman} lemma on dispatchable demand in energy markets}, IEEE
  Transactions on Power Systems \textbf{36} (2021), no.~5, 4028--4041.

\bibitem{Paulin}
P.~{Jacquot}, O.~{Beaude}, S.~{Gaubert}, and N.~{Oudjane}, \emph{Demand
  response in the smart grid: The impact of consumers temporal preferences},
  2017 IEEE International Conference on Smart Grid Communications
  (SmartGridComm), 2017, pp.~540--545.

\bibitem{Jaquotetal2020}
P.~Jacquot, C.~Wan, O.~Beaude, and N.~Oudjane, \emph{Efficient estimation of
  equilibria in large aggregative games with coupling constraints}, IEEE
  Transactions on Automatic Control \textbf{66} (2021), 2762--2769.

\bibitem{Jensen2010}
M.K. Jensen, \emph{Aggregative games and best-reply potentials}, Economic
  theory \textbf{43} (2010), no.~1, 45--66.

\bibitem{Kakutani1941}
S.~Kakutani, \emph{{A generalization of Brouwer’s fixed point theorem}}, Duke
  Mathematical Journal \textbf{8} (1941), no.~3, 457--459.

\bibitem{KerdreuxAl2019}
T.~Kerdreux, I.~Colin, and A.~d'Aspremont, \emph{An approximate
  {Shapley-Folkman} theorem}, arXiv:1712.08559 (2019).

\bibitem{LauerAl1982}
G.S. {Lauer}, N.R. {Sandell}, D.P. {Bertsekas}, and T.A. {Posbergh},
  \emph{Solution of large-scale optimal unit commitment problems}, IEEE
  Transactions on Power Apparatus and Systems \textbf{PAS-101} (1982), no.~1,
  79--86.

\bibitem{LibmanOrda2001}
L.~{Libman} and A.~{Orda}, \emph{Atomic resource sharing in noncooperative
  networks}, Proceedings of INFOCOM '97, vol.~3, 1997, pp.~1006--1013.

\bibitem{MarcottePatricksson2007}
P.~Marcotte and M.~Patriksson, \emph{{Chapter 10. Traffic Equilibrium}},
  Transportation, vol.~14, Elsevier, 2007, pp.~623--713.

\bibitem{Meyers2006}
C.~Meyers, \emph{{Network Flow Problems and Congestion Games: Complexity and
  Approximation Results, PhD dissertation}}, MIT, 2006.

\bibitem{MurphyAl1982}
F.H. Murphy, H.D. Sherali, and A.L. Soyster, \emph{A mathematical programming
  approach for determining oligopolistic market equilibrium}, Mathematical
  Programming \textbf{24} (1982), 92--106.

\bibitem{MyersonWeber1993}
R.B. Myerson and R.J. Weber, \emph{A theory of voting equilibria}, The American
  Political Science Review \textbf{87} (1993), no.~1, 102--114.

\bibitem{OrdaRomShimkin1993}
A.~{Orda}, R.~{Rom}, and N.~{Shimkin}, \emph{Competitive routing in multiuser
  communication networks}, IEEE/ACM Transactions on Networking \textbf{1}
  (1993), no.~5, 510--521.

\bibitem{PaccagnanAl2019}
D.~Paccagnan, B.~Gentile, F.~Parise, M.~Kamgarpour, and J.~Lygeros,
  \emph{{Nash} and {Wardrop} equilibria in aggregative games with coupling
  constraints}, IEEE Transactions on Automatic Control \textbf{64} (2019),
  1373--1388.

\bibitem{PaccagnanKamgarpourLygeros2016}
D.~Paccagnan, M.~Kamgarpour, and J.~Lygeros, \emph{On aggregative and mean
  field games with applications to electricity markets}, 2016 European Control
  Conference (ECC) (2016), 196--201.

\bibitem{PalfreyRosenthal1983}
T.R. Palfrey and H.~Rosenthal, \emph{A strategic calculus of voting}, Public
  Choice \textbf{41} (1983), no.~1, 7--53.

\bibitem{Pappalardo1986}
M.~Pappalardo, \emph{On the duality gap in nonconvex optimization}, Mathematics
  of Operations Research \textbf{11} (1986), no.~1, 30--35.

\bibitem{Rosen1965}
J.B. Rosen, \emph{Existence and uniqueness of equilibrium points for concave
  {N}-person games}, Econometrica \textbf{33} (1965), no.~3, 520--534.

\bibitem{Rosenthal1973}
R.W. Rosenthal, \emph{A class of games possessing pure-strategy {N}ash
  equilibria}, International Journal of Game Theory \textbf{2} (1973), 65--67.

\bibitem{Sagratella2016}
S.~Sagratella, \emph{Computing all solutions of {Nash} equilibrium problems
  with discrete strategy sets}, SIAM Journal on Optimization \textbf{26}
  (2016), no.~4, 2190--2218.

\bibitem{ScutariAl2014}
G.~Scutari, F.~Facchinei, J.~Pang, and D.~Palomar, \emph{Real and complex
  monotone communication games}, IEEE Transactions on Information Theory
  \textbf{60} (2014), 4197--4231.

\bibitem{Selten1970}
R.~Selten, \emph{{Preispolitik der Mehrproduktenunternehmung in der Statischen
  Theorie}}, Springer Verlag Berlin, 1970.

\bibitem{Starr1969}
R.M. Starr, \emph{Quasi-equilibria in markets with non-convex preferences},
  Econometrica \textbf{37} (1969), no.~1, 25--38.

\bibitem{Tran2011}
L.~Tran-Thanh, M.~Polukarov, A.~Chapman, A.~Rogers, and N.R. Jennings, \emph{On
  the existence of pure strategy {N}ash equilibria in integer--splittable
  weighted congestion games}, Algorithmic Game Theory, Springer Berlin
  Heidelberg, 2011, pp.~236--253.

\bibitem{UdellBoyd2016}
M.~Udell and S.~Boyd, \emph{Bounding duality gap for separable problems with
  linear constraints}, Computational Optimization and Applications \textbf{64}
  (2016), 355--378.

\bibitem{VujanicAl2014}
R.~Vujanic, M.E. Peyman, P.~Goulart, and M.~Morari, \emph{Large scale
  mixed-integer optimization: A solution method with supply chain
  applications}, 22nd Mediterranean Conference on Control and Automation
  (2014), 804--809.

\bibitem{VujanicAl2016}
R.~Vujanic, M.E. Peyman, P.J. Goulart, M.~Sebastien, and M.~Manfred, \emph{A
  decomposition method for large scale {MILPs}, with performance guarantees and
  a power system application}, Automatica \textbf{67} (2016), 144--156.

\bibitem{Wang2017}
M.~Wang, \emph{Vanishing price of decentralization in large coordinative
  nonconvex optimization}, SIAM Journal on Optimization \textbf{27} (2017),
  no.~3, 1977--2009.

\bibitem{XuandYin2013}
Y.~Xu and W.~Yin, \emph{A block coordinate descent method for regularized
  multiconvex optimization with applications to nonnegative tensor
  factorization and completion}, SIAM Journal of Imaging Sciences \textbf{6}
  (2013), no.~3, 1758--1789.

\bibitem{YuLui2006}
W.~Yu and R.~Lui, \emph{Dual methods for nonconvex spectrum optimization of
  multicarrier systems}, IEEE Transactions on Communications \textbf{54}
  (2006), no.~7, 1310--1322.

\end{thebibliography}

\end{document}